\documentclass[a4paper,UKenglish,final]{amsart}
\usepackage{color,ulem}

\usepackage{graphicx,fullpage,tikz,tikz-cd,amsmath,amssymb,tikz,comment,mathtools,filecontents,tikzscale,mathdots,enumitem}
\usetikzlibrary{plotmarks} 
\usepackage{subcaption}
\usetikzlibrary{shapes,positioning,intersections,quotes,arrows,matrix}

 \title{Wasserstein Stability for Persistence Diagrams}
\newtheorem{theorem}{Theorem}[section]
\newtheorem{lemma}[theorem]{Lemma}
\newtheorem{corollary}[theorem]{Corollary}

\newtheorem{prop}[theorem]{Proposition}
\newtheorem{definition}[theorem]{Definition}

\newtheorem{remark}[theorem]{Remark}

\newcommand{\R}{\mathbb{R}} % euclidean spaces
\newcommand{\Z}{\mathbb{Z}} % integers
\newcommand{\Dgm}{\mathrm{Dgm}} % persistence diagram
\newcommand{\Xp}{\mathcal{P}} % point set just one
\newcommand{\Xpi}{\mathcal{P}_0} % point set 1
\newcommand{\Ypi}{\mathcal{P}_1} % point set 2

\newcommand{\SC}{K} % simplicial complex
\newcommand{\simplex}{\sigma} %simplex
\newcommand{\Gpt}{v} % point in the pointset
\newcommand{\birth}{\mathbf{b}} % birth time
\newcommand{\death}{\mathbf{d}} % death time
\newcommand{\power}{p} %% wasserstein power
\DeclareMathOperator{\vol}{vol}

\DeclareMathOperator{\supp}{supp}

\newcommand{\PHT}{\operatorname{PHT}}
\newcommand{\CS}{\text{CS}}

\newcommand{\Was}[1]{W_{#1}} % Wasserstein distance
\newcommand{\Mch}{\mathbf{M}} %matching
\newcommand{\cost}{\text{cost}} %correspondence

\newcommand{\Hg}{\mathrm{H}} % homology
\newcommand{\rk}{\mathbf{rk}} % rank function
\newcommand{\Mod}[1]{\mathcal{#1}} %module
\newcommand{\VR}{\mathcal{R}} %vietoris rips
\newcommand{\Diag}{\Delta} % diagonal - do we use this
\newcommand{\D}{\mathcal{D}} %TODO see if we ever use this 

\newcommand{\prs}[1]{#1}
\newcommand{\prsx}[1]{#1}
\newcommand{\kts}[1]{#1}
\newcommand{\kt}[1]{#1}
\newcommand{\pr}[1]{#1}

\begin{document}

\author{Primoz Skraba}
\address{School of Mathematical Sciences, Queen Mary University of London,
	London, UK}
\email{p.skraba@qmul.ac.uk}

\author{Katharine Turner}
\address{Mathematical Sciences Institute, Australian National University,
	Canberra, Australia}
\email{katharine.turner@anu.edu.au}

\begin{abstract}
	The stability of persistence diagrams is among the most important results in applied and computational topology. Most results in the literature phrase stability in terms of the bottleneck distance between persistence diagrams constructed via filtrations by sub-level sets of functions and the $\infty$-norm of perturbations of the input functions. This has two main implications: it makes the space of persistence diagrams rather pathological and it is often provides very pessimistic bounds with respect to outliers. In this paper, we provide new stability results with respect to the $p$-Wasserstein distance between persistence diagrams. This includes an elementary proof for the setting of functions on sufficiently finite spaces (e.g. finite regular CW-complexes) in terms of the $p$-norm of the perturbations, and applying this result to a wide range of applications in topological data analysis (TDA) including topological summaries, persistent homology transforms of shapes and Vietoris-Rips filtrations.
\end{abstract}
\maketitle
\tableofcontents

\section{Introduction}\label{sec:intro}
Persistent homology has been the subject of extensive study in applied topology. Roughly speaking, it is a homology theory for filtrations or filtered spaces. A landmark result is that persistent homology, and more importantly persistence diagrams, % \prs{a complete invariant of persistent homology - NOT true as it ignores bracket type}, 
are stable with respect to perturbations of the input filtration.  The classical result states:
\begin{theorem} \kts{Let $f,g : X \rightarrow \R$ such that all the homology groups of the sublevel sets  of $f$ and $g$ are finitely generated. Then the persistence diagrams  $\Dgm(f)$ and $\Dgm(g)$ for the persistent homology of their sublevel set filtrations satisfy}
	$$d_B(\Dgm(f),\Dgm(g)) \leq ||f -g||_\infty$$
	where $d_B(\cdot) $ represents the bottleneck distance.
\end{theorem}
A version of this result was originally proved for continuous tame functions over a triangulable space in \cite{CohEdeHar2007} and has since been generalized to algebraic~\cite{bauer2014induced} and categorical settings~\cite{bubenik2015metrics}, with recent work strongly aimed at multiparameter and more general settings, particularly where classical notions of persistence diagrams do not exist.
Here we study the $p$-Wasserstein stability of persistence diagrams for $1\leq p \leq \infty$. This has been far less studied,  with existing results almost exclusively in terms of the classical results relating \emph{interleaving distances} between filtrations and the $\infty$-Wasserstein distance,  i.e.  bottleneck distance. Upper bounds on the $p$-Wasserstein distances are \prs{based  primarily on \cite{cohen2010lipschitz}  and  rely} on bottleneck stability resulting in pessimistic bounds. Furthermore, there is often a requirement for $p$ to be sufficiently large for these stability results to hold.
However $p$-Wasserstein distances for small values of $p$ (i.e. $p=1,2$ rather than $p=\infty$) are important with the $2$-Wasserstein distance on persistence diagrams often \prs{being} much more effective than bottleneck distance within applications -- e.g. \cite{7789621, vidal2019progressive,nauleau2022topological,chung2024altered}.  \prsx{In addition to directly using Wasserstein distances between diagrams for data analysis, it has also used to prove the stability of linear representations of persistent homology.} The stability results are usually stated as upper bounds in terms of the $1$-Wasserstein distance, but pre-existing stability results for bottleneck and $p$-Wasserstein distances cannot be applied for this important case of $p=1$. %\prs{The results in this paper address this gap.}
%An important corollary of this paper is that we also get stability results for these linear representations.%In stochastic topology, the bottleneck distance provides bounds which are often too coarse, as they are most often controlled by extreme points. For example, there has been a lot research in understanding the \emph{thermodynamic regime} in geomtric complexes -- where $r=$
%$p$-Wasserstein stability provides a new tool with which to study the topology of probabilistic phenomena.

\emph{Interleavings} are a key tool in existing stability results as they allow us to relate the bottleneck distance between persistence diagrams to the interleaving distance between persistence modules.
One of the main difficulties in establishing a $p$-Wasserstein bound is that interleavings between persistence modules are not sufficient.  %Indeed, at first there does not seem to be a natural algebraic description of Wasserstein distance in the same way as there is for bottleneck distance.
Here we take a fundamentally different approach to proving $p$-Wasserstein stability which, at its core, focuses on a cellular $p$-Wasserstein stability theorem. The proof exploits the local correspondences between coordinates of the points in the persistence diagram with critical cells in a filtration over a cellular complex. This local correspondence \prs{was first used for computing vineyards \cite{cohen2006vines} and more recently for optimization over persistence diagrams \cite{gameiro2016continuation,poulenard2018topological,chen2019topological,leygonie2019framework}}. \prs{A similar technique to the one used in this paper was used in \cite{cohen2006vines} to prove a stability result for the bottleneck distance, but which we use to prove a stability result for the $p$-Wasserstein distance.} This cellular $p$-Wasserstein stability theorem then can applied to a variety of settings to prove a range of stability theorems. 
%

%

%Using the cellular stability theorem as inspiration, we also develop the algebraic theory to deal with the $p$-Wasserstein distance for persistence modules.  The description of the distance is different to the notion of an interleaving which is closer to the single morphism characterization from \cite{bauer2014induced}. Our generalized description recovers this result in the special case of bottleneck distance which is the same as the $\infty$-Wasserstein distance.

%Using this theory, we prove the stability results in a more general setting, albeit with more technically involved proofs compared to the cellular stability proof. We also show a Minkowski-type of result connecting what we call the norm of a persistence modules, also called total persistence, in a short exact sequence\footnote{We note that this is not a norm in the strict sense as persistence modules do not have a vector space structure.}. The upper bound follows naturally from our study of the Wasserstein distance, but surprisingly, there is a lower bound as well, giving a geometric result for extensions of persistence modules for all $p\geq 1$.

%
%While one could make the argument that through the equivalence of norms provide stability results for other values of $p$, it provides only very pessimistic bounds which we improve on.
In summary, the main contributions of this paper are:
\begin{enumerate}
	%  \item We illustrate the limits of stability including several well-known examples,
	\item A cellular $p$-Wasserstein stability theorem for finite complexes.% which also provides a simplified proof of bottleneck stability of diagrams for finite complexes.
	\item The application the above theorem to produce stability theorems for a number of applications, including grey-scale images, persistent homology transforms, and Vietoris-Rips filtrations. As the 2 -Wasserstein distance is widely used in applications, this addresses a significant gap in the applied topology literature.
	\item Discussion of the implications of this $p$-Wasserstein stability for the stability of other representations of persistent homology.
	      %  \item An algebraic formulation of the $p$-Wasserstein distance and prove stability which applies to  pointwise finite dimensional modules with an additional technical condition, 
	      % \item Upper and lower bounds relating the $p$-Wasserstein distance of diagrams of modules in a short exact sequence.
	      %\item Sufficient algebraic conditions which ensure that our algebraic definition yields a pseudo-distance without depending on the structure of persistence modules. 
\end{enumerate}

\section{Preliminaries}\label{sec:preliminaries}
%As described above, we distinguish  between settings where the underlying objects are finite  and more general settings where the objects may be infinite but still sufficiently nice. 
Within this paper we will restrict ourselves to a restricted class of functions over a finite CW-complex $\SC$. This finite setting provides a clear illustration of the ideas and is usually sufficient in applications.
\prs{In the analysis of specific applications, we may assume extra structure} e.g. cubical complexes for images (Section~\ref{sec:images}) and the simplicial structure of Vietoris-Rips complexes (Section~\ref{sec:rips}). %In the algebraic section we consider persistence modules  more generally.

\begin{definition}\label{def:pmodule}
	A persistence module $\Mod{F}$ is a collection of vector spaces $\{F_{\alpha}\}_{\alpha\in\mathbb{R}}$ along with transition maps  $\psi_\alpha^\beta: F_\alpha \rightarrow F_\beta$ for all $\alpha\leq \beta$ such that $\psi_\alpha^\alpha$ is the identity for all $\alpha$ and $\psi_\alpha^\beta\circ \psi_\beta^\gamma=\psi_\alpha^\gamma$ whenever $\alpha<\beta<\gamma$.
	If $F_\alpha$ is finite dimensional for all $\alpha$, then we say $\Mod{F}$ is pointwise finite dimensional (or p.f.d.).
\end{definition}

The building blocks of persistence modules are interval modules.

\begin{definition}
	Let $A \subset \R$ be an interval, and fix a field $\mathbb{F}$. The \emph{interval module} with support $A$, is the persistence modules such that $F_\alpha=\mathbb{F}$ for $\alpha\in A$ and $F_x=0$ for $\alpha\notin A$, and such that the transition maps $\phi_\alpha^\beta=\text{id}$ for $\alpha<\beta$ both in $A$ and the zero map otherwise. We denote this interval module by $\mathcal{I}_A$.
\end{definition}

\prs{When the} persistence module is pointwise finite dimensional (p.f.d.), which will be the case throughout this paper, we may apply the following theorem.
\begin{theorem}[\cite{Crawley_Boevey_2015} Theorem 1.1]\label{thm:intervaldecomposition}
	A p.f.d. persistence module admits an \emph{interval decomposition}. That is, the module is isomorphic to a direct sum of interval modules over some index set $S$:
	$$\bigoplus\limits_{x\in S}  \mathcal{I}_{A_x}$$
	which are unique up to reordering of $S$. %This is referred to as a \emph{persistence barcode} and we refer to the elements as bars or intervals.
\end{theorem}

Throughout this paper we will only ever see intervals with finite infimums so for the sake of clarity we will restrict to persistence modules of this form. The generalisation to consider $-\infty$ as the beginning of the intervals does not change any of the arguments significantly.
We can associate to each interval module $\mathcal{I}_A$ a point in 
$$ \overline{\mathbb{R}}^{2+}:=\{(a,b)\in \R\times \{\R\cup \infty\}\mid a\leq b\}$$ with the first coordinate $\birth(\mathcal{I}_A):=\inf(A)$ and second coordinate $\death(\mathcal{I}_A):=\sup(A)$. We refer to $\birth(\mathcal{I}_A)$ as the birth time and $\death(\mathcal{I}_A)$ as the death time. \prs{Taking the collection of pairs associated to the interval decomposition} of a p.f.d. persistence module $\Mod{F}$ we obtain a multiset of points in $\overline{\mathbb{R}}^{2+}$. We refer to this multiset as the \emph{persistence diagram}, and denote it $\Dgm(\Mod{F})$.

One of the most common ways persistence modules arise is via filtrations of finite CW-complexes. A \emph{filtration} of a topological space $K$ is a parameterised family of subsets $\{K_\alpha\subseteq K|\alpha\in \R\}$ such that $K_\alpha\subseteq K_\beta$ whenever $\alpha\leq \beta$. In this paper we will assume our filtrations arise as sublevel sets over a CW-complex: for
$f:\SC\rightarrow \mathbb{R}$, with
\begin{enumerate}
	\item\kts{ $f$ constant on the interior of each cell,}
	      %From the definition, it is clear we only consider functions which are piecewise constant on the interior of cells, that is
	      %$$f(\sigma) = \sup\limits_{x\in \sigma} f(x).$$
	\item  $f$ is \emph{monotone}, i.e., \kts{ if} $\tau$ is a face of $\sigma$ then $f(\tau)\leq f(\sigma)$.
\end{enumerate}
The monotone assumption ensures that all sublevel sets are (closed) CW-complexes.  Hence, we have corresponding sublevel set filtration $\{K_\alpha\}_{\alpha\in \R}$ with
$$\SC_\alpha = \{\sigma | f(\sigma) \leq \alpha\}.$$
This is the most common setting in applications of persistence.
%We note that although this is more restrictive than the definition in ~\cite{EdeHar2010}, which only includes the condition that the space monotonically non-decreasing. 
The \kts{ assumption that the functions are constant on each cell} excludes piecewise linear (PL) functions over simplicial complexes.
\kts{However, if $f$ is a piecewise linear function on a simplicial complex $S$, where the function on each cell is liner interpolation of the values on the vertices, then we can construct a monotone function $f' : S \to \R$, with persistent homology isomorphic to that of $f$. We just need to take each simplex value under $f'$ to be the maximum value of $f$ over its vertices.}
%However, for finite CW complexes one can always find a  refinement of the CW-complex and a piecewise constant function over this refined CW-complex which results in an isomorphic persistence module to the module arising from the PL function. 

\kts{The monotonicity condition implies that the sublevel set filtration induces} a filtered cellular chain complex. By applying the homology functor over a field to \kts{that} filtered chain complex, we obtain the corresponding persistence module, denoted $\{\Hg_k(\SC_\alpha)\}_{\alpha\in\mathbb{R}}$ with the transition maps  $\psi_\alpha^\beta:\Hg_k(\SC_\alpha) \rightarrow \Hg_k(\SC_\beta)$, induced by the inclusions  $\SC_\alpha \hookrightarrow \SC_\beta$ for all $\alpha\leq \beta$.

%One could generalize the results in Section~\ref{sec:simplicial_stability} to a more general setting, such as constructible functions. However the increase in technical complications yields relatively little gain in practice. 
As we restrict ourselves to piecewise constant filtration over a finite CW-complex, the resulting persistence module is pointwise finite dimensional (p.f.d.) and so we can apply Theorem \ref{thm:intervaldecomposition}. \kts{For a monotone function $f:K \to \R$ we will use $\Dgm_k(f)$ to denote the persistence module for the degree-$k$ persistence module for the sublevel set filtration of $f$, and use $\Dgm(f)$ to denote the union of the $\Dgm_k(f)$.}

Our main focus is the Wasserstein distance between diagrams. The Wasserstein distance is a form of optimal transport metric. We can consider all possible transportation plans for moving the points within one persistence diagram to a different one. The transportation plans between persistence diagrams are called matchings. Each of these transportation plans has a cost and the distance becomes the infimum of these costs. To define our potential transportation plans we need to introduce an abstract element which can be thought of as an empty interval. We call this abstract element the \emph{diagonal}  \prs{and} denote it by $\Diag$. 

\begin{definition}
	%Set $\overline{\R}$ to be the extension of the real line to include $\pm \infty$ and set $\R^{2+}=\{(b,d)\in \overline{\R}^2\mid b\leq d\}$. 
	Let $\Dgm(\Mod{F})$ and $\Dgm(\Mod{G})$ be persistence diagrams represented by countable multisets in $\overline{\R}^{2+}$. A \emph{matching} $\Mch$ between $\Dgm(\Mod{F})$ and $\Dgm(\Mod{G})$ is a subset of $\{ \Dgm(\Mod{F}) \cup \Diag\} \times \{\Dgm(\Mod{G})\cup \Diag\}$ \kts{such that the number of elements of the form $(x, \cdot)$ is equal to the multiplicity of $x$ in $\Dgm(\Mod{F})$   and the number of elements of the form $(\cdot, y)$ is equal to the multiplicity of $y$ in  $\Dgm(\Mod{G})$ .} %such that every element in $X$ and $Y$ appears in exactly one pair (and for elements in $X$ and $Y$ with higher multiplicity the number of times it appears in $\Mch$ is the same as its multiplicity). 
The abstract diagonal element $\Diag$ may appear in many pairs.
\end{definition}

As the points in $\Dgm(\Mod{F})$ and $\Dgm(\Mod{G})$  lie in $\R^2$ we can use the $l_p$ distance in the plane. With a slight abuse of notation, we can define an $l_p$ distance between a point in $\R^2$ to $\Diag$ by taking the perpendicular distance, that is
$$\left \|(a,b)-\Diag\right \|_p=\inf_{t\in \R}\left \|(a,b)-(t,t)\right\|_p=\left\|(a,b)-\left(\frac{a+b}{2}, \frac{a+b}{2}\right)\right\|_p=2^{\frac{1-p}{p}} |b-a|$$
and $\|(a,b)-\Diag\|_\infty=\frac{|b-a|}{2}$.
Furthermore we say for all $p$ that $\|\Diag-\Diag\|_p=0$, $\|(a,\infty)-(b, \infty)\|_p=|a-b|$ and $\|(a,\infty)-x\|_p=\infty$ for $x\in \R^2\cup \Delta$.
%We then define the $p$-cost of matching $\Mch\subset \{X\cup \Diag\}\times \{Y\cup \Diag\}$ as $$c_p(\Mch)=\sum_{(x,y)\in \Mch}\|x-y\|_p.$$

%\pr{In defining the distance between diagrams, points on the diagonal, i.e. $\birth(\Dpt)=\death(\Dpt)$, must be added. The set of all points such that birth equals death  is denoted by $\Diag$.

%As the points lie in $\R^2$, there are natural associated norms, e.g. the 2-norm is the standard Euclidean norm. 
%These apply to all pairs of points, except if both points lie in $\Diag$, in which case the distance between the pair is 0.}

We define the \emph{p-cost} of a matching by taking the sum over all the pairs within the matching and taking to the appropriate power, that is
$$\cost_p(\Mch)= \left(\sum\limits_{(x,y)\in \Mch} ||x-y||_p^p\right)^{\frac{1}{p}}.$$

\begin{definition}\label{def:wasserstein}
	Given two persistence diagrams, $\Dgm_k(\Mod{F})$ and $\Dgm_k(\Mod{G})$, the $p$-Wasserstein distance is
	$$\Was{p}(\Dgm_k(\Mod{F}), \Dgm_k(\Mod{G})) = \inf\limits_{\Mch} \cost_p(\Mch)$$
	where $\Mch\subset \{\Dgm_k(\Mod{F})\cup \Delta\} \times \{\Dgm_k(\Mod{G})\cup \Delta\}$ is a matching.  The total $p$-Wasserstein distance is defined as
	%$$\Was{p,q}(\Dgm(\SC(f)), \Dgm(\SC'(g))) = \left(\sum_{k}\left(\Was{p,q}(\Dgm_k(\SC(f)), \Dgm_k(\SC'(g)))\right)^p\right)^{\frac{1}{p}}$$
	$$\Was{p}(\Dgm(\Mod{F}), \Dgm(\Mod{G})) = \left(\sum_{k}\left(\Was{p}(\Dgm_k(\Mod{F}), \Dgm_k(\Mod{G})\right)^p\right)^{\frac{1}{p}}$$
\end{definition}

\begin{remark}\label{rem:wasserstein}
	It is worth noting that there is some discrepancy in th	e literature in the definition of the Wasserstein distance between persistence  diagrams.  More generally, given two diagrams, $\Dgm_k(\Mod{F})$ and $\Dgm_k(\Mod{G})$, we could define a $(p,q)$-Wasserstein distance as
	$$\Was{p,q}(\Dgm_k(\Mod{F}), \Dgm_k(\Mod{G})) = \inf\limits_{\Mch} \left(\sum\limits_{(x,y)\in \Mch} ||x-y||_q^p\right)^{\frac{1}{p}} $$
	where $\Mch\subset \{\Dgm_k(\Mod{F})\cup \Delta\} \times \{\Dgm_k(\Mod{G})\cup \Delta\}$ is a matching.  The total $(p,q)$-Wasserstein distance is defined as
	%$$\Was{p,q}(\Dgm(\SC(f)), \Dgm(\SC'(g))) = \left(\sum_{k}\left(\Was{p,q}(\Dgm_k(\SC(f)), \Dgm_k(\SC'(g)))\right)^p\right)^{\frac{1}{p}}$$
	$$\Was{p,q}(\Dgm(\Mod{F}), \Dgm(\Mod{G})) = \left(\sum_{k}\left(\Was{p,q}(\Dgm_k(\Mod{F}), \Dgm_k(\Mod{G})\right)^p\right)^{\frac{1}{p}}$$
	%This is a standard approach to defining the distances and for more details, we refer the reader to \cite{mileyko2011probability,turner2014frechet, turner2020medians}. 
	For any fixed $p$, all the $\Was{p,q}$ as $q$ varies are bi-Lipschitz equivalent. However, we restrict ourselves to the case of  $p=q$. This will give the best bounds for the stability results in this paper. It also gives (Fr\'echet) means and medians of collections of persistence diagrams a nicer characterisation (see \cite{turner2014frechet, turner2020medians}) than other choices of $q$.
\end{remark}

%The main case of interest in applications and the only one we consider in this paper is $p=q$, which is denoted as $\Was{\power}$.
Taking the limit $p\rightarrow\infty$ recovers the bottleneck distance
\begin{equation}
	%  \Was{\infty}(\Dgm(\SC(f)), \Dgm(\SC'(g))) = \adjustlimits\inf_{\mathbf{M}} \sup_k \sup_{\Dpt\in \Dgm_k(\SC(f)) } ||\Dpt -\Mch(\Dpt)||_\infty
	\Was{\infty}(\Dgm(\Mod{F}), \Dgm(\Mod{G})) = \sup_k \inf_{\Mch_k} \sup_{(x,y)\in \Mch_k} ||x-y||_\infty
\end{equation}
where $\Mch \subset \{\Dgm_k(\Mod{F})\cup \Delta\} \times \{\Dgm_k(\Mod{G})\cup \Delta\}$ is a matching.
It is worth commenting on the relative strength of stability results for different $p$. 
 We first note the following lemma\footnote{This is a well-known result but we could not find an appropriate reference so the  proof in the appendix is included for completeness.} whose proof can be found in  Appendix~\ref{sec:powers}.
\begin{lemma}\label{lem:powers}
	For any $\power'\leq \power$, given persistence diagrams  $\Dgm(\Mod{F})$ and $\Dgm(\Mod{G})$, $$\Was{\power}(\Dgm(\Mod{F}),\Dgm(\Mod{G}))\leq \Was{\power'}(\Dgm(\Mod{F}),\Dgm(\Mod{G})).$$
\end{lemma}
This lemma implies that when bounding the $\power$-Wasserstein distance from above, the smaller $\power$ is, the stronger the stability result. %Hence, the bottleneck distance is the least sensitive distance. 

%  \prs{Q: CHECK - DO WE USE THE NORM NOTATION ANYWHERE BESIDE THE ALGEBRAIC section.}

We will want to consider the space of persistence diagrams as a metric space with the $p$-Wasserstein metric for different values of $p\in [1,\infty]$. To do this we will want to restrict the set of allowable persistence diagrams in a manner similar to restricting the space of functions to those that are integrable.

\begin{definition}
	For persistence diagram $\Dgm(\Mod{F})$ let $\Dgm(\Mod{F})_{finite}$ be the subset of $\Dgm(\Mod{F})$  with finite coordinates. The space of persistence diagrams is the set of persistence diagrams   $\Dgm(\Mod{F})$ such that $\sum_{(b,d)\in \Dgm(\Mod{F})_{finite}}|b-d|<\infty$ and $\Dgm(\Mod{F})\backslash \Dgm(\Mod{F})_{finite}$ has finitely many elements. We denote this space $\D$. 
\end{definition}

Note that every persistence diagram constructed by sublevel set filtrations of a function over a finite CW complex will be contained in $\D$. The $p$-Wasserstein distance becomes an extended metric over $\D$.

\section{Existing stability results and their limitations}\label{sec:existing_limitations}
%We need to state the current results and corollaries, being very clear as to the stability they are relying on. In particular list all the ``stability'' claims which use 1-stability of the diagrams.
As already mentioned, almost all stability results involve the bottleneck distance between persistence diagrams. While a complete overview of these results is beyond the scope of this paper, \prs{key references include the original stabilty result \cite{CohEdeHar2007}, and its algebraic and categorical generalizations in \cite{chazal2012structure} and ~\cite{Bubenik_2014} respectively.} Additionally,  stability results
have been been shown for specific constructions, e.g.  geometric complexes ~\cite{chazal2014persistence}.
%Beginning with   and we direct the reader to \cite{chazal2014persistence} for the stability of geometric constructions and ~\cite{Bubenik_2014} for the categorical foundations of $\infty$-Wasserstein stability. 
This should  not be considered as a complete list as there is a large body of work on stability (for distances other than Wasserstein distance) which we do not review here.
%\begin{itemize}
%  \item Stability of geometric constructions~\cite{chazal2014persistence}%
%  \item Categorical formulations of stability~\cite{Bubenik_2014}
%  \item Stability for functionals~\cite{}
%  \item Stability for probabilistic constructions~\cite{}
%\end{itemize}

%\color{red}
%List some versions of bottleneck stability, sublevel sets (functions over same domain or %simplicial complex), Rips/Cech
%\color{black}

\subsection{Lipschitz functions on compact manifolds}
%NEED TO EDIT NOW HAVE MORE THOUGH STILL STEMMING FROM BOUNDING TOTAL PERSISTENCE
The work most related to the results presented in this paper can be found in  ~\cite{cohen2010lipschitz}.
To the best of our knowledge, this paper contains the main existing stability result for bounding the ($p\neq \infty$)-Wasserstein distance between two persistence diagrams. It is for the setting of sub-level set filtrations of Lipschitz functions. \prs{This stability result depends on a quantity called \emph{degree $k$ total persistence.}}
\begin{definition}[\cite{cohen2010lipschitz}]
	A metric space $X$ \emph{implies bounded degree $k$-total persistence} if there exists a constant $C_X$ that depends only on $X$ such that
	$$\sum_{x\in \Dgm(f)}\|x-\Delta\|_k^k < C_X $$
	for every tame function $f$ with Lipschitz constant $Lip(f)\leq 1$.
\end{definition}
\kts{It is proven in ~\cite{cohen2010lipschitz} that sublevel-set filtrations of tame Lipschitz functions on triangulable, compact metric spaces that imply bounded degree-$k$ total persistence enjoy $p$-Wasserstein stability for $p>k$.}

\begin{theorem}[Wasserstein Stability Theorem~\cite{cohen2010lipschitz}] \label{thm:cohen}
	\prs{Let $X$ be a triangulable, compact metric space that implies bounded degree-$k$ total persistence, for some real number $k\geq 1$, and let $f,g:X\to \R$ be two tame Lipschitz functions. Then $$W_p(\Dgm(f),\Dgm(g))\leq C^{1/p}\|f-g\|_\infty^{1-\frac{p}{k}}$$ for all $p\geq k$, where $C=C_X \max\{Lip(f)^k, Lip(g)^k\}$ and $C_X$ is a constant dependent on $X$.}
\end{theorem}

To put our results into context, it is worthwhile understanding the limitations of this theorem. We will find lower bounds on $C_X$ and $k$, restricting ourselves to the case where $X$ is a compact $d$-dimensional manifold.  An important aspect is the requirement that the domain implies bounded degree-$k$ total persistence which will force this stability result to only hold for sufficiently large $p$.

To construct a counterexample for functions over manifolds we will use a function which is the sum of functions with supports over disjoint balls of small radius.

%, and use only $f,g$ with $Lip(f), Lip(g)\leq 1$.
\begin{lemma}\label{lem:packing}
	Given an $d$-dimensional compact Riemannian manifold $X$ and $r>0$ small enough, there exists a packing of $\left\lfloor \frac{\vol(X)}{\kappa \omega_d 2^d r^d}\right\rfloor$ disjoint balls of radius $r$ in $X$, where $\omega_d$ is the volume of the $d$-dimensional Euclidean unit ball and $\kappa$ is a constant which depends on the infimum of the scalar curvature of $X$.
\end{lemma}
\begin{proof}
	\prsx{As we assume $X$ is compact, its scalar curvature is bounded from below. Let $s_{min}$ denote this lower bound. 
	 Corollary 3.2 in \cite{bobrowski2019random} upper bounds the volume of a ball of radius $r$ by $ \omega_d r^d (1-(s_{min}+o(1))r^2 )$. For our purposes, we may simplify this to $\kappa \omega_d r^d$ for small enough $r$ and a constant $\kappa$ which depends on $s_{min}$.
	 The result then follows from standard arguments relating packing and covering numbers.  The covering number with balls of radius $r$ must be at least $\vol(X)/\sup \vol(B_r)$ by a volume argument. Additionally,  the packing number with balls of radius $r$ is lower bounded by the covering number with balls of radius $2r$. Substituting the upper bound on the volume for balls of radius $2r$ gives the result.}
\end{proof}
% The proof then follows from standard arguments involving packing and covering numbers.

\begin{lemma}\label{lem:teepee}
	Let $X$ an $d$-dimensional compact \kts{Riemannian} manifold. If $X$ implies bounded degree-$k$ total persistence then $k\geq d$.
\end{lemma}
\begin{proof}
	We will prove this via a counterexample when $k<d$. \kts{Since $X$ is compact there is a global lower bound $\rho$ on the injectivity radius \cite[Theorem 2.1.10]{klingenberg1995riemannian}.  Choose $0<r<\rho$ small enough so that by Lemma \ref{lem:packing} there exists  $N=\left\lfloor \frac{\vol(X)}{\kappa\omega_d 2^d r^d}\right\rfloor$ disjoint balls of radius $r$ in $X$. Let $P=\{p_1, \ldots p_N\}$ be the centers of these balls.}
	% of $N=\left\lfloor \frac{\vol(X)}{\kappa\omega_d 2^d r^d}\right\rfloor$ disjoint balls of radius $r$ in $X$ (such a packing exists by Lemma \ref{lem:packing}).
	Set $T_{r,p}$ to be a teepee shaped function about $p$ with height $r$, with $T_{r,p}(x)=\max\{r-d(x,p), 0\}$.
	We then consider functions $f_r=\sum_{i=1}^N T_{r,p_i}$(see Figure~\ref{fig:teepee}).  Observe that $f$ is 1-Lipschitz. Then,
	$$||\Dgm(f)||^k_k  = \sum_{i=1}^N r^k = \left\lfloor \frac{\vol(X)}{\kappa\omega_d 2^d r^d}\right\rfloor r^k = \Theta(r^{k-d}).$$
	When $k<d$ then this cannot be uniformly bounded from above for all small enough $r>0$.
\end{proof}
\begin{figure}[th]
	\centering
	\begin{tikzpicture}[scale=0.6,every node/.style={scale=0.8}]
		% Draw axes
		\draw [<->,thick] (-2.4,0) -- (2.4,0) node (xaxis) [at end, below] {$x$};
		\draw [->,thick] (0,-0.1) -- (0,2) node [at start, below] {$p$};
		\draw (-0.1,1.5) -- (0.1,1.5) node (peak) [at start,left] {$r$};
		\draw (-1.5,0) -- (0,1.5) -- (1.5,0);
		\draw (-1.5,-0.1) -- (-1.5,0.1) node (l) [at start, below] {$p-r$};
		\draw (1.5,-0.1) -- (1.5,0.1) node (r) [at start, below] {$p+r$};
		% Draw two intersecting lines
	\end{tikzpicture}\hspace{1cm}
	\begin{tikzpicture}[scale=0.6,every node/.style={scale=0.8}]
		% Draw axes
		\draw [<->,thick] (-8,0) -- (8,0) node (xaxis) [at end, below] {$x$};
		\draw [->,thick] (0,-0.1) -- (0,2) node [at start, below] {$p_i$};
		\draw (-0.1,1.5) -- (0.1,1.5) node (peak) [at start,left] {$r$};
		\draw (-1.5,0) -- (0,1.5) -- (1.5,0);
		\draw (-1.5,-0.1) -- (-1.5,0.1) node [at start, below] {$p_i-r$};
		\draw (1.5,-0.1) -- (1.5,0.1) node  [at start, below] {$p_i+r$};

		\draw  (5,-0.1) -- (5,0.1) node [at start, below] {$p_{i+1}$};
		\draw (3.5,0) -- (5,1.5) -- (6.5,0);
		\draw (3.5,-0.1) -- (3.5,0.1) node  [at start, below] {$p_{i+1}-r$};
		\draw (6.5,-0.1) -- (6.5,0.1) node  [at start, below] {$p_{i+1}+r$};

		\draw  (-5,-0.1) -- (-5,0.1) node [at start, below] {$p_{i-1}$};
		\draw (-3.5,0) -- (-5,1.5) -- (-6.5,0);
		\draw (-3.5,-0.1) -- (-3.5,0.1) node  [at start, below] {$p_{i-1}+r$};
		\draw (-6.5,-0.1) -- (-6.5,0.1) node  [at start, below] {$p_{i-1}-r$};

		\node (ld) at (-7,0.75) {$\mathbf{\cdots}$};
		\node (rd) at (7,0.75) {$\mathbf{\cdots}$};
	\end{tikzpicture}

	\caption{\label{fig:teepee} (Left) The teepee function and (right) sum of teepee functions from Lemma~\ref{lem:teepee}.}
\end{figure}
Lemma \ref{lem:teepee} is easily extended to more general spaces, such as stratified spaces.
If we compare the teepee function from Lemma \ref{lem:teepee} to the zero function on the same space, Lemma \ref{lem:teepee} allows us to deduce  a lower bound on the value of $C_X$ in Theorem \ref{thm:cohen}.  Observe that the teepee function has  an sup-norm of $r$. Substituting this into the bound from Theorem \ref{thm:cohen}, for small enough $r$ we obtain that  $C^{1/p}_X$  grows linearly as a function of the volume of $X$. That is, we have
$$ C_X^{1/p} \geq \left \lfloor \frac{\vol(X)}{\kappa\omega_d 2^d r^d}\right \rfloor r^{k-1+p/k}\geq  \frac{\vol(X)}{\kappa\omega_d 2^{d+1}} r^{k-1+p/k- d}$$

Although our counterexample for bounded degree $k$-total persistence is only for homology in degree $n-1$ (for a manifold of dimension $n$), analogous counterexamples can be made for homology in other degrees. For example, by taking the negative we have a counterexample with degree $0$ homology. For other homology degrees, we can use different local functions such as one with support
on an annulus rather than a ball.

\prs{Besides \cite{cohen2010lipschitz}, there are very few Wasserstein stability results in the literature.}
%Other papers that prove  are few and far between. 
In \cite{chen2011diffusion}, Chen and Edelsbrunner consider non-Lipschitz functions on non-compact spaces, using scale-space diffusion. They focus on convergence properties as opposed to stability but also attain some stability results for the $p$-norm of the diagram (Wasserstein distance to the empty diagram). Crucially, just as in the Lipschitz case, this $p$-Wasserstein stability only holds for $p > d$ where $d$ is the dimension of the domain. The condition $p > m$ also appears in stability results for \v Cech filtrations, or equivalently distance filtrations, for point clouds sampled on an $m$-dimensional submanifold of $\R^d$ \cite{arnal2024wasserstein}.
%
% \color{red}
%
% Mention paper that does applies to related problem on Cech for point cloud in $\R^n$ but again only for $p>n$. They leave Rips as open problem which we partially address in this paper
% \color{black}

\subsection{Erroneous appeals to previous $p$-Wasserstein stability results}
Unfortunately, the Lipschitz Wasserstein stability theorem in \cite{cohen2010lipschitz} appears to be one of the most misunderstood and miscited results within the field of topological data analysis. Common errors include using a small $p$ (often $1$ or $2$) for high dimensional data, assertions that the persistence diagrams depend Lipschitz-continuously on data and applying the theorem to Vietoris-Rips filtrations. Luckily, many of the erroneous applications can now be covered by the stability results in this paper. Rather than discuss individual examples, in Section~\ref{sec:consequences} we examine the consequences of stability for various topological summaries including vectorisations such as the persistent homology rank function.
%
% Unfortunately, the Lipschitz stability theorem in \cite{cohen2010lipschitz} appears to be one of the most misunderstood and miscited reults within the field of topological data analysis.
%
%
% ``The essence of the stability results, as shown in [3, 6, 9], is that the persistence diagrams depend Lipschitz-continuously on point cloud data.''\cite{berwald2014critical}
% ``Assume that the time series (tj , xj ) is obtained as a time discretization of a process (t, xt) which is Lipschitz continuous.''
% ``Diagrams corresponding to nearby times will be close to one another, due to the Lipschitz continuity of the process underlying the time series and to the robustness of persistence diagrams. ''
%
%
% Thankfully, many of the erroneous applications can now be covered by the stability results in this paper.

\section{Cellular Wasserstein Stability}\label{sec:simplicial_stability}
We begin with a result mirroring the classical stability theorem by bounding the differences at the chain level, which will induce an upper bound on the $\power$-Wasserstein distance between the corresponding diagrams. As stated in Section~\ref{sec:preliminaries}, we remind the reader that $\SC$ is a finite CW-complex and $f:\SC \to \R$ is a monotone function on the complex, so all sublevel sets are subcomplexes and that the persistence module associated is p.f.d.  
We begin by defining an $L^p$ norm of a function on $\SC$.
\begin{definition}
	The $L^\power$ norm of a function $f: \SC\rightarrow \R$ is given by
	$$\|f\|_\power^\power=\sum_{\simplex\in \SC}|f(\simplex)|^\power.$$
\end{definition}

This induces a distance between functions. The $L^\power$ distance between two monotone functions $f,g: \SC\rightarrow \R$ is given by
$\|f-g\|_\power^\power. $
% \begin{definition}
%  The $L^\power$ distance between two monotone functions $f,g: \SC\rightarrow \R$ is given by
%  $$\|f-g\|_\power^\power = \sum_{\simplex\in \SC}|f(\simplex)-g(\simplex)|^\power.$$
% \end{definition}sss
Note that this notions of the $L^\power$ distance between functions is analogous to the $L^\power$ distance for functions over discrete sets where  the sum here is over the discrete set of cells. \prsx{In this section, we prove the cellular Wasserstein stability.}
\begin{lemma}\label{lem:noswitch}
	Let $f,g:\SC\to \R$ be monotone functions over a CW complex $\SC$ such that the partial orders induced by $f$ and $g$ embed into a common total order, then
	$$\Was{\power}(\Dgm(f),\Dgm(g))\leq ||f-g||_\power$$
	and
	$\Was{\power}(\Dgm_k(f),\Dgm_k(g))^p\leq \sum_{\dim(\sigma)\in\{k, k+1\}} |f(\sigma)-g(\sigma)|^p.$
\end{lemma}

%\pr{It remains an open problem to determine the relationship between the simplicial norm and the more classical functional $p$-norm or to Sobelov norms.}
%\begin{problem}
%\pr{Do there exist reasonable} conditions  on the underlying space and the function to relate the cellular $p$-norm to the more common functional $p$-norm, e.g. the integral of $f^p$ over the space?
%\end{problem}

% \begin{remark}
% 	Given that we restrict ourselves to piecewise constant monotone functions, the only additional condition we require is that the underlying complex is finite. We have chosen to present the results in this way, so that it is clear that it applies to common settings including simplicial and cubical complexes.
% \end{remark}
The main idea in the proof is to bound the Wasserstein distance by considering a straight line homotopy between $f$ and $g$. We split the straight line homotopy into finitely many sub-intervals where a local result will hold, and then collect together the summands for the final desired inequality. By focusing on small enough sub-intervals we can exploit a consistent correspondence between the coordinates of the points in the persistence diagram with critical cells in the filtration.  
As noted in the introduction, this idea was previously used in \cite{cohen2006vines} to
show a bottleneck stability result for functions on simplicial complexes.
%Though we phrase the proof in a different way, this idea first appeared in \cite{cohen2006vines} to show a bottleneck stability result \prs{using} ideas from vineyards which were introduced in the same paper. Indeed, the idea of tracking of points in the persistence diagram is fundamental to the definition of vineyards. We refer the reader to \cite{cohen2006vines} for more details and an algorithmic perspective on tracking critical simplices.

\kts{ We will be using the language of partial and total orders to describe when we can consistently label birth and death times of points in the persistence diagrams by the same choice of simplices. }

\begin{definition}
	\kts{ A \emph{partial order} $P\subset X\times X$ over a set $X$ is a binary relation ($<$) that is irreflexive ($(a,a)\notin P$ for all $a$), asymmetric ($(a,b)\in P$ implies $(b,a)\not\in P$), and transitive ($(a,b),(b,c)\in P$ implies $(a,c)\in P$). A \emph{total order} is a partial order which is connected (for $a\neq b$ either $(a,b)$ or $(b,a)$ must be in $P$). We say that $P$ embeds into $Q$ if $P\subset Q$.}
\end{definition}

There is a natural partial order we can construct from a monotone function. This partial order records when one cell must always appear before another one. 

\begin{definition}\label{def:partialorder}
	$f:\SC\to \R$ be a monotone function. We define the partial order $P_f$ on $\SC$ induced by $f$ by $(\sigma, \tau)\in P_f$ (with $\sigma \neq \tau$) whenever
	$f(\sigma) < f(\tau)$, or
	$f(\sigma) = f(\tau)$ and $\sigma \subset \tau$.
\end{definition}
We first consider the easy case by bounding the distance between functions where the ordering of cells does not change. \prsx{Typically, the partial order induced by $f$ is not a total order.} However, it will be convenient to embed the partial order into a total order, which can always be done for finite partial orders. We say two partial orders can be embedded into a \emph{common total ordering} if the total order contains both partial orders.
\prsx{We state the following proposition for completeness.}
	\begin{prop}\label{prop:persistence_algorithm}
		\prsx{After extending the partial order $P_f$ to a total order $Q$ there is a unique partition of the underlying complex into pairs} $(\sigma, \tau)$ (such that $(\sigma, \tau)\in Q$) and singletons $(\omega, \emptyset)$ such that
		\begin{enumerate}
			\item each cell appears exactly once,
			\item the points in the persistence diagram are $(f(\sigma), f(\tau))$ and $(f(\omega), \infty)$.
		\end{enumerate}
	\end{prop}
	This follows directly from the \emph{persistence algorithm} e.g. \cite{edelsbrunner2000topological,zomorodian2005computing,cohen2006vines}, where the interval decomposition is computed by finding a pairing of cells. Briefly, each cell either creates a new interval, i.e., generates a new homology class, at $f(\sigma)$ or $f(\omega)$, or ends an interval, i.e., bounds an existing class, at $f(\omega)$. %Since an injective filtration function corresponds to a total ordering and 
Given a total order each insertion of a cell induces a rank 1 change to the homology \cite{delfinado1993incremental}, which allows us to deduce uniqueness. We also remark that the result may be shown using the matroidal properties of the cycle and boundary bases, see ~\cite{skraba2017randomly}. See Appendix \ref{appendix:persistence} for further details.

% }
% \prs{ We recall three standard facts about persistence diagrams when the filtration is a total order.
% 	\begin{enumerate}
% 		\item The birth times for points in a $k$-persistence diagram correspond to the filtration values of $k$-simplices, and finite death times correspond to  $(k+1)$-simplices.\label{it:pairing}
% 		\item This correspondence is unique and determined by the total ordering.
% 		\item \label{it:unique}
% 	\end{enumerate}
% 	\ref{it:pairing} May be deduced along the lines of~\cite{delfinado1993incremental}, as changes in the homology may only occur when a cell is added to the filtration, but follows directly from the \emph{persistence algorithm} \cite{edelsbrunner2000topological,zomorodian2005computing,cohen2006vines}, where this correspondence is used to compute the interval decomposition. \ref{it:unique} may also be deduced from the persistence algorithm, but also follows immediately from \cite[Lemma 25]{skraba2017randomly}.}

\begin{proof}[Proof of Lemma \ref{lem:noswitch}]
	By assumption, we may fix one total order $Q$ such that the partial orders $P_f$ and $P_g$ are both contained in $Q$. \prs{By Proposition \ref{prop:persistence_algorithm}, we may partition the cells in $K$ into pairs $(\sigma_i, \tau_i)$ and singletons $(\omega_j, \emptyset)$ such that each cell appears exactly once.} %The algorithm for computing persistent homology for $f$ and $g$ partitions the cells in $K$ into pairs $(\sigma_i, \tau_i)$ and singletons $(\omega_j, \emptyset)$ such that each cell appears exactly once. 
	The partition of the cells for $f$ and $g$ will be the same as it depends only on the total order $Q$. \prs{Additionally,} the intervals of the persistence modules for $f$ and $g$ are then given by
	$\{[f(\sigma_i), f(\tau_i)): f(\tau_i)>f(\sigma_i)\} \cup \{[f(\omega_j), \infty)\}$ and  $\{[g(\sigma_i), g(\tau_i): g(\tau_i)>g(\sigma_i)\} \cup \{[g(\omega_j), \infty)\}$ respectively. %Note that by construction of the total order,  we have $ f(\sigma_i)\leq  f(\tau_i))$ for all $i$.

	To bound the Wasserstein distance between $\Dgm_k(f) $ and $\Dgm_k(g)$ we will consider the transportation plan where we are match $(f(\sigma_i), f(\tau_i))$ with $(g(\sigma_i), g(\tau_i))$ and $(f(\omega_j),\infty)$ with $(g(\omega_i), \infty)$. Let $A=\{i:f(\tau_i)>f(\sigma_i) \text{ and }g(\tau_i)>g(\sigma_i)\}$, $B=\{i: f(\tau_i)>f(\sigma_i) \text{ and }g(\tau_i)=g(\sigma_i)\}$ and $C=\{i: f(\tau_i)=f(\sigma_i) \text{ and }g(\tau_i)>g(\sigma_i)\}$.
	Construct a matching $\Mch\subset ( \Dgm(f)\cup \Delta) \times (\Dgm(g) \cup \Delta)$ given by the union of the multisets
	\begin{align*}
		 & \{((f(\sigma_i), f(\tau_i)), (g(\sigma_i), g(\tau_i))): i\in A\} \\
		 & \{((f(\sigma_i), f(\tau_i)), \Delta):i \in B\}                   \\
		 & \{(\Delta, (g(\sigma_i), g(\tau_i))): i\in C\}                   \\
		 & \{((f(\omega_j),\infty), (g(\omega_i), \infty))\}.
	\end{align*}

	Note that for $i\in B$ we have $\|(f(\sigma_i), f(\tau_i))- \Delta\|^p_p \leq |f(\sigma_i)-g(\sigma_i)|^p +|f(\tau_i))-g(\tau_i)|^p$ as $g(\sigma_i)=g(\tau_i)$, and for $i\in C$ we have $\|\Delta -(g(\sigma_i), g(\tau_i))\|^p_p \leq |f(\sigma_i)-g(\sigma_i)|^p +|f(\tau_i))-g(\tau_i)|^p$ as $f(\sigma_i)=f(\tau_i)$.
	The $\power$-th power of the cost of this matching $\Mch$ is thus bounded by
	$$\sum_\sigma |f(\sigma)-g(\sigma)|^p$$
	as every cell appears at most once.
	% \begin{align*}
	%  \sum_{\Dpt} & \|(f_a(\simplex(\birth_c(\Dpt))),f_a(\simplex(\death_c(\Dpt))))-(f_{a'}(\simplex(\birth)_c(\Dpt))),f_{a'}(\simplex(\death_c(\Dpt))))\|_p^p                 \\
	%           & =\sum_{x} |f_a(\simplex(\birth_c(\Dpt))) - f_{a'}(\simplex(\birth_c(\Dpt)))|^\power+ \sum_{\death(\Dpt)\neq \infty } |f_a(\simplex(\death_c(\Dpt))) - f_{a'}(\simplex(\death_c(\Dpt)))|^\power \\
	%           & \leq \sum_{\simplex\in \SC} |f_a(\simplex)-f_{a'}(\simplex)|^\power
	%            =\|f_a-f_{a'}|\|_\power^\power
	% \end{align*}
	% Note that we are using the facts that the distance to the diagonal is bounded by the distance to any specific point on the diagonal, and that every cell appears at most once in the middle sum.
	%
	Since the $\power$-Wasserstein distance is the smallest possible matching cost, we conclude that
	$$\Was{\power}(\Dgm(f), \Dgm(g))^\power\leq \|f-g\|_\power^\power.$$
	%
	%If cells have coinciding function values over the interval $(a,a')$ then we  observe that changes in homology at that function value must be caused by one of the set of cells. However, as the bound on the $\power$-Wasserstein distance only uses the function values which are equal for all cells in the set over the entire interval, the distance is independent of the choice of critical cell.
	%%
	%
	The proof for when we restrict to homology dimension $k$ follows from the observation that only $k$-cells and the $k+1$-cells can affect homology in dimension $k$.
\end{proof}

We are now ready to proof the general case of Cellular Wasserstein Stability.

\begin{theorem}[Cellular Wasserstein Stability Theorem]\label{thm:simpbound}
	Let $f,g:\SC\to \R$ be monotone functions. Then
	\begin{enumerate}
		\item[(i)] $\Was{\power}(\Dgm(f),\Dgm(g))\leq \|f-g\|_\power,$
		\item[(ii)]	$\Was{\power}(\Dgm_k(f),\Dgm_k(g))^p\leq \sum_{\dim(\sigma)\in\{k, k+1\}} |f(\sigma)-g(\sigma)|^p.$
	\end{enumerate}
\end{theorem}
\begin{proof}
The main idea in the proof is to bound the Wasserstein distance by considering a straight line homotopy between $f$ and $g$. Let $f_t:\SC\to\R$ be the linear interpolation between $f$ and $g$ as $t$ varies. That is, for $t\in [0,1]$ and $\simplex \in \SC$, let $f_t(\simplex)=(1-t)f(\simplex)+tg(\simplex).$
	Observe that $f_t$ is monotone for all $t$ and that for \kts{$0\leq t\leq t'\leq 1$} we have
	$\|f_t-f_{t'}\|_p=|t'-t|\|f-g\|_\power$.

	Each of the functions $t\mapsto f_t(\simplex)$ is linear which implies that $f_t(\simplex)=f_t(\hat{\simplex})$ for two or more values of $t$ if and only if $f(\simplex)=f(\hat{\simplex})$ and $g(\simplex)=g(\hat{\simplex})$, in which case $f_t(\simplex)=f_t(\hat{\simplex})$ for all $t\in [0,1]$).

\begin{figure}[htbp]
	\centering\includegraphics[width=0.4\textwidth,page=1]{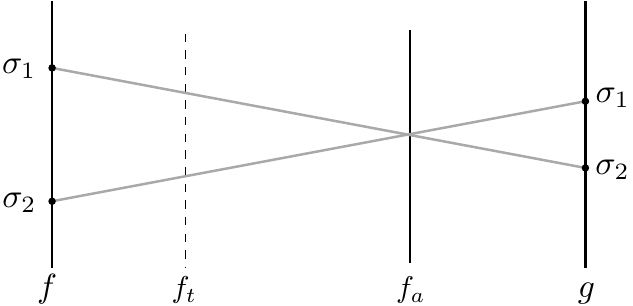}
	\caption{\label{fig:interp} A linear interpolation between functions $f$ and $g$ can be subdivided into intervals where the ordering does not change in the interior of the interval. If the underlying space is a finite CW complex, the number of such intervals is finite.}
\end{figure}

Observe that this straight line homotopy may be divided into intervals where the ordering does not change, see Figure~\ref{fig:interp}. Since our underlying space is a finite CW complex, the number of such intervals must also be finite. There are only finitely many values $t=a_1,a_2, \ldots a_n$ in $(0,1)$, sorted in increasing value, where there exists $\simplex, \hat{\simplex}$ with $f_t(\simplex)=f_t(\hat{\simplex})$ but $f(\simplex)\neq f(\hat{\simplex})$. Set $a_0=0, a_{n+1}=1$. In each of the intervals $t\in(a_i, a_{i+1})$, \prs{all $f_t$ induce the same partial ordering}. Hence, there exists a \prs{common} total ordering which is compatible with all the induced partial orders on $\SC$ for any choice in $(a_i, a_{i+1})$. This choice of total ordering will also contain the partial orders induced by $f_{a_i}$ and $f_{a_{i+1}}$. As noted above, if two simplices have equal function values, at any point in the open interval, they have equal function values over the entire interval.
	%Therefore, a valid total ordering at any point remains valid over the entire interval. 
	We can therefore apply Lemma~\ref{lem:noswitch} for with $f=f_{a_i}$ and $g=f_{a_{i+1}}$.

	%\begin{align*}
	%W_p(\Dgm(f),\Dgm(g))&\leq \sum_{i=0}^n W_p(\Dgm(f_{a_i}),\Dgm(f_{a_{i+1}}))\\
	%&\leq  \sum_{i=0}^n \|f_{a_i}-f_{a_{i+1}}\|_p\\
	%&=\sum_{i=0}^n(a_{i+1}-a_i) \|f-g\|_p\\
	%&=\|f-g\|_p
	%\end{align*}

	\begin{align*}
		\Was{\power}(\Dgm(f),\Dgm(g)) & \leq \sum_{i=0}^n \Was{\power}(\Dgm(f_{a_i}),\Dgm(f_{a_{i+1}}))                                           \\
		                              & \leq  \sum_{i=0}^n \|f_{a_i}-f_{a_{i+1}}\|_\power=\sum_{i=0}^n(a_{i+1}-a_i) \|f-g\|_\power=\|f-g\|_\power
	\end{align*}

	The proof for (ii) %when we restrict to homology dimension $k$ 
	is analogous, using the corresponding bound in Lemma \ref{lem:noswitch} but restricting to homological dimension $k$.
\end{proof}

\section{Applications}\label{sec:applications}
In this section we will present some applications of the results of the cellular Wasserstein stability theorem. Sublevel set filtrations of grayscale images and persistent homology transforms of different geometric embeddings of the same simplicial complex are both cases which involve height functions determined by vertex values. We will prove Lipschitz stability in terms of the $l_p$ norms over the set of vertices, where the Lipschitz constants are bounded by the number of cells in the links of each vertex.
We also will prove some immediate corollaries for stability of Rips filtrations.

There are different notions of points we consider; points in the persistence diagrams and elements of a point set in $\R^d$. To minimize confusion, we restrict the term points to refer to persistence diagrams, preferring the term vertices for the more geometric notions. While this has the drawback of resulting in references to vertices in a point set, we feel this is a good compromise. %For a complete list of notation, see Appendix~\ref{sec:notation}.

\subsection{Stability of the sublevel set filtrations of grayscale images}\label{sec:images}
Our first application is \kts{to} the stability of \kts{the persistent homology of} grayscale images. \kts{While most applications would be for two and three dimensional images,} we will state our results for more general $d$-dimensional images. An image is a real-valued piecewise constant function where each pixel/voxel is assigned a value. There are two main methods in the literature for creating a filtration of cubical complexes from a grayscale image.

\subsubsection*{Method 1}
We can create a cubical complex from a 2D image where each pixel corresponds to a 2-dimensional cubical cell. The edges correspond to sides of the pixels, and vertices to the corners. This construction naturally extends to higher dimensional images. There is a natural sublevel set filtration induced on the complex: the image defines values for the maximal dimensional cells (i.e. pixels/voxels) and the function values for lower dimensional cells are given as the minimum value over all cofaces.

\subsubsection*{Method 2}
We can also consider the dual of the cubical complex in Method 1, which is again a cubical complex. In a 2D image we have a vertex for each pixel and an edge for each pair of neighbouring pixels (not including diagonals), and 2-cells where four pixels intersect. This construction naturally extends to higher dimensional images. We can build a filtration on this cubical complex by setting the values on the vertices as those of the pixel/voxel values provided, and setting the function values for higher dimensional cells as the maximum value over all faces.\\

It is worth noting that the sublevel set filtrations for these two methods can result in drastically different persistent homology. This difference stems from whether diagonally neighbouring pixels are considered connected. However, applying Theorem~\ref{sec:simplicial_stability} separately to each method obtains stability for both methods individually. \prs{While Method 1, is far more common in TDA applications, e.g. \cite{robins2011theory, tymochko2020using,dunaeva2016classification}, Method 2 often appears as the dual complex of an image \cite{bleile2022persistent}.}
\begin{theorem}
	Let $f$ and $g$ be the grayscale functions of two images \pr{of the same dimensions} \kt{over the same grid of pixels}. Let $\hat{f}$ and $\hat{g}$ be the corresponding monotone functions on the underlying cubical complex generated by either Method 1 or 2 (both $\hat{f}$ and $\hat{g}$ using the same method).  Then
	$$\Was{p}(\Dgm(\hat{f}),\Dgm(\hat{g})) \leq \left(\sum\limits_{i=0}^{d} 2^{d-i} \binom{d}{i}\right) ||f-g||_p   $$
\end{theorem}
\begin{proof}
	Let us suppose we are using Method 1 for  constructing  our persistence diagrams.
	As the underlying space is a cubical complex, changing the function value of a maximal cell can affect all of the lower dimensional cells it contains. Each $d$-dimensional hypercube contains  $2^{d-k} \binom{d}{k}$ $k$-dimensional hypercubes on its boundary.  Summing up over all dimensions yields a bound on how many cell-values change when we change the value of a pixel. Applying Theorem~\ref{thm:simpbound} yields the result.

	The proof for Method 2 is similar. Changing the function value of a vertex can affect all of the higher dimensional cofaces. There are at most $2^{k} \binom{d}{k}$ possibly affected $k$-dimensional cells and applying Theorem~\ref{thm:simpbound} completes the proof.
\end{proof}

\subsection{Stability of persistent homology transforms} \label{sec:transforms} %of different embeddings of the same simplicial complex}
The study of persistent homology transforms are a relatively recent development in the persistent homology literature~\cite{turner2014persistent,ghrist2018persistent,curry2018many} with applications to statistical shape analysis.
The most general setting of the PHT is for constructible sets which are compact definable sets \kt{(see \cite{curry2018many})}.  %We denote the set of constructible subsets of $\R^d$ by $\CS(\R^d)$. 
In this paper, we consider a smaller class of sets of finite embedded geometric simplicial complexes.

\begin{definition}
	Let $K$ be a finite simplicial complex with vertex set $V$. For $f:V\to \R^d$ we can define a piece-wise linear extension of $f:K\to \R^d$ by setting $f(\sum a_i v_i)=\sum a_i f(v_i)$. We call $f:K\to \R^d$ a \emph{geometric vertex embedding} of $K$ if $f(K)$ is a geometric realisation of $K$ (i.e. no self-intersections). A subset $M\subset \R^d$ is \emph{finite embedded geometric simplicial complex} if it is the image of a geometric vertex embedding of a finite simplicial complex.
\end{definition}

Given a shape $M\subset \R^n$, every unit vector $v$ corresponds to a height function in direction $v$,
\begin{align*}
	h_v & :M \to \R                      \\
	h_v & :x\mapsto \langle x, v\rangle.
\end{align*}
where $\langle\cdot,\cdot\rangle$ denotes the inner product.
The resulting degree-$k$ persistence diagram computed by filtering $M$ by the sub-level sets of $h_v$, is denoted $\Dgm_k(h^M_v)$.
This diagram records geometric information from the perspective of direction $v$. As $v$ changes, the persistent homology classes track geometric features in $M$. \kts{ The key insight behind the persistent homology transform (PHT) is
	%we do not need to choose a specific direction; 
	that by considering the persistent homology from every direction we obtain a retain complete information about the shape.}
%

%which are the embedded.
%\kts{Here we restrict ourselves to different embeddings of the same underlying simplicial complex, where the embeddings are each determined linearly by the placement of the vertices. We call such an embedding a vertex embedding.}
 
	%, which we refer to simply as geometric complexes.}

%\sout{All compact piece-wise linear sets are constructible \kt{and this set of shapes is sufficient for the purposes of this paper.}}% \pr{See  \cite{van1998tame} for a formal definition and treatment of $\CS(\R^d)$. TODO: Is this the right reference?} 

\begin{definition}
	The \emph{Persistent Homology Transform} $\PHT$ of a %\sout{compact piece-wise linear set} 
	\prs{finite embedded geometric simplicial complex} $M$ is the map $\PHT(M): S^{d-1} \to \Dgm^d$ with
	%that sends a direction to the set of persistent diagrams gotten by filtering $M$ in the direction of $v$:
	\[
		\PHT(M): v \mapsto \left(\Dgm_0(h_v^M),\Dgm_1(h_v^M), \ldots, \Dgm_{d-1}(h_v^M)\right)
	\]
	where $h_v^M:M\to \R$, $h_v^M(x)=\langle x, v\rangle$ is the height function on $M$ in direction $v$.
	Letting the set $M$ vary gives us the map
	\[
		\PHT : \CS(\R^d) \to C^0(S^{d-1},\Dgm^d),
	\]
	where $C^0(S^{d-1},\Dgm^d)$ is the set of continuous functions from $S^{d-1}$ to $\Dgm^d$, the latter being equipped with some Wasserstein $p$-distance.
\end{definition}

The persistent homology transform is a complete descriptor of constructible sets; for $M_1, M_2 \subset \R^d$, $\PHT(M_1)=\PHT(M_2)$ implies  $M_1=M_2$ as subsets of $\R^d$. This was originally proved in~\cite{turner2014persistent} for $\R^2$ and $\R^3$, and then the more general proof was given in~\cite{curry2018many} and independently in~\cite{ghrist2018persistent}. Finite embedded geometric simplicial complexes are examples of constructible sets.
%

%Here we restrict ourselves to different embeddings of the same simplicial complex where the embeddings are each determined linearly by the placement of the vertices.
%We call such shapes the geometric vertex embedding of a finite simplicial complex.% and notably they are always a constructible set.

We can define a metric on the space of persistent homology transforms by considering the appropriate integrals of Wasserstein distances in each direction. We obtain a different distance for each $p\in [1,\infty]$

\begin{definition}
	For $p\in [1,\infty)$, and sets $M_1, M_2 \subset \R^d$ \prs{whose PHTs are defined, the}  $p$-PHT distance between $M_1, M_2$ is defined as
	$$d_p^{\PHT}(M_1, M_2) = \left(\int_{S^{d-1}} W_p(\Dgm(h_v^{M_1}),\Dgm(h_v^{M_2}))^p \,dv\right)^{1/p}.$$
\end{definition}

We can use the cellular Wasserstein stability result to prove a stability theorem for the persistent homology transforms of different vertex embeddings of the same simplicial complex.

\begin{theorem}
	Fix a simplicial complex $K$ with vertex set $V$. Let  $C_{p,d}=2\omega_{d-2}\int_0^{\frac{\pi}{2}}\cos^p(\theta)  \sin^{d-2}(\theta)\, d\theta$ where $\omega_{d-2}$ is the volume of the unit sphere $S^{d-2}$ and
	$$C_K=\max_{\text{vertices }v\in K}|\{\sigma\in K|v\in \overline{\sigma}\}|.$$
	Let $f,g:K \to \R^d$ be different geometric vertex embeddings of $K$. Then
	$$d_{p}^{\PHT}(f(K), g(K))\leq \left(C_K C_{p,d} \sum_{v\in V} \|f(v)-g(v)\|_2^p\right)^{1/p}.$$
\end{theorem}
\begin{proof}
	Define functions $k^f_w:K\to \R$ by setting
	$$k^f_w([v_0, \ldots v_n]) = \max\{h_w(f(v_0), h_w(f(v_1)), \ldots h_w(f(v_n))\},$$ and $k^g_w:K\to \R$ \kt{similarly}.
	As discussed in \cite{curry2018many}, the sublevel set filtrations of $k_w^f$ and $h_w^{f(K)}$ have the same persistent homology. Similarly, $k^g_w$ and $h_w^{g(K)}$ give the same sub-level set persistent homology.
	By Theorem~\ref{thm:simpbound}, we know that
	$$W_p(\Dgm(k^f_w), \Dgm(k^g_w))^p\leq \sum_{\Delta\in K}|k_w^f(\Delta)-k_w^g(\Delta)|^p.$$
	For any finite set $X$,
	$$\left|\max_{x\in X}f(x)-\max_{y\in X} g(y)\right| \leq \max_{x\in X}\left|f(x)-g(x)\right|$$
	which implies
	$$\sum_{\sigma\in K}\left|k_w^f(\sigma)-k_w^g(\sigma)\right|^p\leq \sum_{\sigma\in K}\max_{v\in \sigma}\left\{\left|k_w^f(v)-k_w^g(v)\right|^p\right\}\leq C_K\sum_{v\in V} \left|k_w^f(v)-k_w^g(v)\right|^p.$$

	But $k_w^f(v)=\langle w,f(v)\rangle$ and $k_w^g(v)=\langle w, g(v)\rangle$ which implies $$\sum_{\sigma\in K}\left|k_w^f(\sigma)-k_w^g(\sigma)\right|^p\leq C_K\sum_{v\in V} \left|\langle w, f(v)-g(v)\rangle\right|^p.$$

	\begin{align*}
		d_p^{\PHT}(f(K), g(K))^p & =\int_{S^{d-1}} W_p(\Dgm(h^f_w), \Dgm(h^g_w))^p\, dw                      \\
		                         & \leq \int_{S^{d-1}}C_K \sum_{v\in V} |\langle w,f(v)-g(v)\rangle |^p\,dw  \\
		                         & \leq C_K \sum_{v\in V} \int_{S^{d-1}} |\langle w,f(v)-g(v)\rangle |^p\,dw
	\end{align*}
	Let $e_1$ denote the vector with $1$ in the first coordinate and $0$ in every other coordinate. For each $u\in \R^d$ we have $ \int_{S^{d-1}} |\langle w,u\rangle |^p\,dw=\|u\|_2^p \int_{S^{d-1}}|\langle w, e_1\rangle |^p\,dw.$ This implies
	\begin{align*}
		d_p^{\PHT}(f(K), g(K))^p & \leq C_K \sum_{v\in V} \| f(v)-g(v)\|_2^p\int_{S^{d-1}}|\langle w, e_1\rangle |^p\,dw                                                                        \\
		                         & \kts{=C_K \sum_{v\in V} \| f(v)-g(v)\|_2^p 2 \,\int_{\theta=0}^{\pi} \int_{\{w\in S^{d-1}: \langle w, e_1\rangle=\cos(\theta)\}} |\cos(\theta)|^p \,dw \,dt} \\
	\end{align*}
	\kts{For $\theta\in  (0,\pi)$ the set $\{w\in S^{d-1}: \langle w, e_1 \rangle=\cos(\theta)\}$ is a scaled version of $S^{d-2}$ with radius $\sin(\theta)$. We can use the symmetry between $\theta$ and $\pi-\theta$ to remove the need for absolute value signs. The inequality thus becomes	}
	\begin{align*}
		d_p^{\PHT}(f(K), g(K))^p & \leq C_K \sum_{v\in V} \| f(v)-g(v)\|_2^p \,2 \omega_{d-2} \,\int_{0}^{\frac{\pi}{2}} \cos(\theta)^p \sin^{d-2}(\theta) \,d\theta. \\
	\end{align*}

\end{proof}
In particular $C_{p,3}=\frac{4\pi}{p+1}$, $C_{p,2}\leq 2$ for all $p$, and $C_{1,d}=\frac{2\omega_{d-2}}{d-1}$.

\subsection{Stability results for Rips complexes}\label{sec:rips}
One of the most common uses of persistent homology is that of filtrations of Rips complexes over point clouds (where a \emph{point cloud} is just a finite set of points in Euclidean space). The goal of this section is to bound the change in $\Dgm(\VR(\Xp))$ as the underlying point cloud $\Xp$ changes, so we first find the appropriate distance between point clouds. We will first state the definition of the Wasserstein distances between measures. This views each point cloud as a sum of point masses. In order for this distance to be defined we require that the point clouds have same cardinality. If there were different numbers of points then the total masses of the measures are different and no transport plan can be formed between the two measures. \prs{While one could normalize the measure to avoid this issue, the resulting problem would involve comparing the limits of persistence diagrams, which is a largely open problem and beyond the scope of this paper.}

\begin{definition}\label{def:ptwasserstein}
	Let  $\Xpi$ and $\Ypi$ be two finite point clouds in $\R^d$ and assume $|\Xpi| =|\Ypi|$.  Define the \emph{point cloud Wasserstein distance} between them as
	$$\Was{p}^{\text{point cloud}}({\Xpi},{\Ypi}) = \inf\limits_{\phi}  \left(\sum\limits_{\Gpt\in \Xpi } ||\Gpt-\phi(\Gpt)||^\power\right )^{\frac{1}{\power} }.$$
	where $\phi$ is a bijection.
\end{definition}

Since we are dealing with finite sets \prs{of equal cardinality} this definition is equivalent to the classical Wasserstein distance between the measures $\mu_0$ and $\mu_1$ where $\mu_i=\sum_{x\in \Xp_i}\delta_x$. \prs{This equivalence is well-known as the $p$-Wasserstein distance may be expressed as a linear program, e.g. \cite{peyre2019computational}. This linear program is known to have an integral solution\footnote{This follows from a reduction to a minimum cost maximum flow problem.}  see \cite[Chapter 13]{schrijver2003combinatorial}. }

Before stating our stability results let us first recall some basic definitions.

\begin{definition}\label{def:vrcomplex}
	Given a point cloud $\Xp \subset \R^d$, the Vietoris-Rips complex \kts{at length scale $\delta$} is the simplicial complex $\VR_\delta(\Xp)$ where a $k$-simplex is a subsets of $k+1$ points $\{\Gpt_1,\ldots, \Gpt_{k+1}\}$ such that $||\Gpt_i - \Gpt_j||_2\leq \delta$ for all $i,j=1,\ldots, k+1$.
\end{definition}
%We implicitly use the identification of the vertices of $\VR(\Xp)$ and the points of $\Xp$.

\kts{We can define a Vietoris-Rips function such that the sublevel set for value $\delta$ is the Vietoris-Rips complex at length scale $\delta$.}
%By varying $\delta$, we obtain a filtration.

\begin{definition}\label{def:vrfiltration}
	\kts{The Vietoris-Rips function of a point cloud $\Xp$ is the function $\VR(\Xp):K \to \R$ where $K$ is the compete simplical complex over $\Xp$ and $\VR(\Xp)([v_{i_0}, v_{i_1}, \ldots v_{i_k}])=\max_{j,l}\{\|v_{i_j}-v_{i_l}\|_2\}$.}
	%	 fil $\{\VR_\delta(\Xp)\}$ induced by ranging $\delta$ from 0 to $\infty$. The corresponding persistence diagram is denoted $\Dgm(\VR(\Xp))$.
\end{definition}
%\begin{definition}\label{def:vrfiltration}
%	The Vietoris-Rips filtration (or simply Rips filtration) of a point cloud $\Xp$ is the filtration $\{\VR_\delta(\Xp)\}$ induced by ranging $
%\delta$ from 0 to $\infty$. The corresponding persistence diagram is denoted $\Dgm(\VR(\Xp))$.
%\end{definition}

%\kts{Note that the sublevel set filtration of the Vietoris-Rips function is usually shortened to the Vietoris-Rips filtration.}

\begin{theorem}
	Fix $M>0$. For all $\power\geq 1$, for all $k$, and all point clouds $\Xpi, \Ypi$ with $|\Xpi|,|\Ypi|=M$ we have
	$$\Was{\power}(\Dgm_k(\VR(\Xpi)),\Dgm_k(\VR(\Ypi)) ) \leq 2\binom{M-1}{k}^{1/p}\Was{p}^{\text{point cloud}}({\Xpi},{\Ypi}).$$
	%	where $\Dgm_k(\VR(\Xpi))$  and $\Dgm_k(\VR(\Ypi))$ are the $k$-dimensional persistence diagrams for the Vietoris-Rips filtration on the point clouds $\Xpi$ and $\Ypi$ respectively.
	Furthermore, $$\Was{\power}(\Dgm(\VR(\Xpi)),\Dgm(\VR(\Ypi)) ) \leq 2^{M/p+1} \Was{\power}^{\text{point cloud}}(\Xpi,\Ypi).$$
\end{theorem}
\begin{proof}
	Let $\phi:\Xpi\to \Ypi$ be a bijection which achieves the minimum of
	$$W_p^{\text{point cloud}}(\Xpi, \Ypi)=\inf\limits_{\phi}  \left(\sum\limits_{\Gpt \in \Xpi } ||\Gpt-\phi(\Gpt)||^\power\right )^{\frac{1}{\power}}.$$
	Relabel the points in $\Xpi=\{x_1, \ldots x_M\}$ and $\Ypi=\{y_1, \ldots y_M\}$ so that $\phi(x_i)=y_i$.
	Let $K$ be the complete simplicial complex on $M$ vertices $\{v_1, \ldots v_M\}$.
	\kts{	Define functions $f, g: K \to \R$ by setting
		$f([v_{i_0}, v_{i_1}, \ldots v_{i_k}])=\VR(\Xpi)([x_{i_0},x_{i_1}, \ldots x_{i_k}]) $ and
		$g([v_{i_0}, v_{i_1}, \ldots v_{i_k}])=\VR(\Ypi)([y_{i_0}, y_{i_1}, \ldots y_{i_k}]$. That is we are precomposing the Vietoris-Rips functions with the appropriate bijection of vertices. By construction $\Dgm_k(f)=\Dgm_k(\Xpi)$ and $\Dgm_k(g)=\Dgm_k(\Ypi)$.
	}
	Suppose for now that $k\geq 1$. Then
	\begin{align*}
		|f([v_{i_0}, v_{i_1}, \ldots v_{i_k}])-g([v_{i_0}, v_{i_1}, \ldots v_{i_k}])|
		 & = |(\max_{j, l}\{ \|x_{i_j}- x_{i_l}\|\}- \max_{j,l}\{ \|y_{i_j}- y_{i_l}\|\}| \\
		 & \leq \max_{j, l} |\|x_{i_j}-x_{i_l}\|- \|y_{i_j}-y_{i_l}\||.
	\end{align*}

	By the triangle inequality $|\|x_{i_j}-x_{i_l}\|- \|y_{i_j}-y_{i_l}\||<\|x_{i_j}-y_{i_j}\|+ \|x_{i_l}-y_{i_l}\|$. This implies $|f([v_{i_0}, v_{i_1}, \ldots v_{i_k}])-g([v_{i_0}, v_{i_1}, \ldots v_{i_k}])|\leq \max_{j\neq l} \|x_{i_j}-y_{i_j}\|+ \|x_{i_l}-y_{i_l}\|\leq 2 \max_{j} \|x_{i_j}-y_{i_j}\|$.

	Since $K$ is the complete simplicial complex over $M$ vertices, each edge $[v_i, v_j]$ appears in $\binom{M-2}{k-1}$ $k$-simplices (we only need to decide which extra $k-1$ vertices to include).

	Using the cellular stability theorem,
	\begin{align*}
		\Was{\power} & (\Dgm_k(\VR(\Xpi)),\Dgm_k(\VR(\Ypi)) )^p                                                                                                                                                                          \\
		             & \leq \sum_{[v_{i_0},  \ldots v_{i_k}]} |f([v_{i_0},  \ldots v_{i_k}])-g([v_{i_0},  \ldots v_{i_k}])|^p +\sum_{[v_{i_0},  \ldots v_{i_{k+1}}]} |f([v_{i_0},  \ldots v_{i_k}])-g([v_{i_0},  \ldots v_{i_{k+1}}])|^p \\
		             & \leq \sum_{i} \binom{M-2}{k-1} 2^p\|x_i-y_i\|^p+ \sum_{i} \binom{M-2}{k} 2^p\|x_i-y_i\|^p                                                                                                                         \\
		             & \leq 2^p\binom{M-1}{k}\Was{p}^{\text{point cloud}}({\Xpi},{\Ypi})^p                                                                                                                                               \\
	\end{align*}

	For $k=0$ the calculations are even easier as the vertex values are all $0$. %By the cellular stability theorem,
	\begin{align*}
		\Was{\power} (\Dgm_0(\VR(\Xpi)),\Dgm_0(\VR(\Ypi)) )^p & \leq \sum_{i<j} |f([v_i, v_j])-g([v_i, v_j])|^p   \\
		                                                      & =\sum_{i<j} |\|x_i-x_j\|-\|y_i-y_j\||^p           \\
		                                                      & \leq \sum_{i} (2\|x_i-y_i\|)^p                  \\
		                                                      & =2^p\Was{p}^{\text{point cloud}}({\Xpi},{\Ypi})^p
	\end{align*}

	To prove the second part, we again use the cellular stability theorem to compute
	\begin{align*}
		\Was{\power}(\Dgm_k(\VR(\Xpi)),\Dgm_k(\VR(\Ypi)) )^p & \leq \sum_{k=1}^M \sum_{[v_{i_0}, v_{i_1}, \ldots v_{i_k}]} |f([v_{i_0}, v_{i_1}, \ldots v_{i_k}])-g([v_{i_0}, v_{i_1}, \ldots v_{i_k}])|^p \\
		                                                     & \leq \sum_{k=0}^M \binom{M}{k} 2^p \sum_{i} \|x_i-y_i\|^p                                                                                   \\
		                                                     & = 2^p2^M \Was{p}^{\text{point cloud}}({\Xpi},{\Ypi})^p.                                                                                     \\
	\end{align*}
\end{proof}

\section{Consequences for  topological summaries}\label{sec:consequences}
\prsx{Summary statistics based on topological invariants are commonly known as \emph{topological summaries}. A basic example of a topological summary is the space of persistence diagrams equipped with any of the $p$-Wasserstein metrics.} One drawback of persistence diagrams is that they do not form a Hilbert space, and so it is difficult to use them with standard statistical or machine learning techniques. A common approach to overcome this is to consider topological summaries which are derived from persistence diagrams but are more amenable to additional processing. Often these can contain the same information as persistence diagrams (such as persistence images) or strictly less information (such as Betti curves), so we can consider them as the output of a function from the space of persistence diagrams. 

A desirable property of a topological summary is stability. In the literature, \prsx{stability results for topological summaries} are often stated in terms of the
	$p$-Wasserstein distance of the corresponding persistence diagrams, typically using $p=1$. The $1$-Wasserstein distance provides the largest upper bound on the distance between topological summaries amongst the Wasserstein metrics, and cannot be controlled by $p$-Wasserstein distances for $p>1$.
	In contrast, most bounds on the distance between
persistence diagrams generated from of geometric input, such as point clouds or metric spaces, are upper bounds the bottleneck
distance in terms of geometric quantities such as Hausdorﬀ(-type) distances.
	As the bottleneck distance is a lower bound for the $p$-Wasserstein distance, for all $p$, bottleneck stability results are not easily combined with 1-Wasserstein stability results for topological summaries.
	%In contrast, most bounds on the distance between persistence diagrams are in terms of  geometric measures of the difference in the input. For example, Hausdorff distance between point clouds or perturbations of input functions, are almost always considered with respect to the bottleneck distance. As this is the smallest lower bound on the distance between diagrams, these results cannot be easily combined.

Here we apply our results to give a bound directly on the 1-Wasserstein distances for  topological summaries where the condition of $d(T(X),T(Y))\leq C_T W_1(X,Y)$ for all persistence diagrams $X,Y$ has already been established. This includes
\begin{enumerate}
	\item sliced Wasserstein kernel, $C_T=2\sqrt{2}$ \cite{carriere2017sliced},
	\item persistent images, $C_T=1$ \cite{adams2017persistence},
	\item persistent scale space, $C_T=1$ \cite{reininghaus2015stable,kusano2016persistence},
	\item Betti curves \cite{robins2006betti, saadat2021topological},
	\item learned/optimized representations \cite{hofer2017deep,hofer2019learning},
	\item persistent homology rank function \cite{robins2016principal}, $C_T=1$ (Corollary \ref{cor:rank}).
\end{enumerate}
While many of the results have stability results in terms of the input, the result below provides a common setting which we believe will be useful for new summaries. Additionally, the stability some summaries, e.g. the rank function, lack a previous stability result in the literature.

To this end, we first examine  how  Lipschitz stability relates to linear representations of persistence diagrams, providing necessary conditions. Finally, we also consider persistence landscapes \cite{bubenik2015statistical} which are one of the most common forms of non-linear representations. We prove negative Lipschitz stability results for all $L^p$ function norms of persistence landscapes where $p<\infty$.
%%%REMEMBER TO UNCOMMENT THE BELOW
% \begin{definition}
% 	Let $\D$ be the space of persistence diagrams equipped with any   $p$-Wasserstein metric for $p\geq 1$.
% \end{definition}
We begin with the general result:

	\begin{corollary}
		\prsx{Let $(\mathcal{X},d)$ be a  metric space and  $T: \D \to \mathcal{X}$ a function.} 
		Suppose that there exists a $C_T>0$ such that $d(T(X),T(Y))\leq C_T W_1(X,Y)$ for all persistence diagrams $X,Y$.
		If $f,g$ are monotone functions over cellular complex $K$ then %, with $T(\Dgm(f))$ and $T(\Dgm(g))$ 
		%the corresponding topological summaries for the sub-level filtrations of $f$ and $g$ respectively then
		$d(T(\Dgm(f)), T(\Dgm(g)))\leq C_T\|f-g\|_{1}$.
		%where $\|f-g\|_1=\sum_{\sigma\in K}|f(\sigma)-g(\sigma)|$.
	\end{corollary}

The proof follows directly from the earlier stability results in this paper.

\subsection{Linear representations of persistence diagrams}\label{sec:linear}
\prs{ Linear representations of persistence diagrams are a common form of topological summaries. Viewing persistence diagrams as measures over the plane, i.e.  assuming a  persistence diagram $X$ has off-diagonal points $\{x_i\}$, its corresponding \emph{persistence measure} is $\mu_X=\sum_{i}\delta_{x_i}$ where $\delta_{x_i}$ is the Dirac measure on $x_i$. Then any function from the plane to some Banach space gives a resulting linear representation via integration over the persistence measure.}

% A common form of topological summaries are linear representations of persistence diagrams. Examples of linear representations include persistence images, persistent rank functions and weighted Betti curves. We view persistence diagrams as measures over the plane. \kts{More precisely, if persistence diagram $X$ has off-diagonal points $\{x_i\}$ its corresponding \emph{persistence measure} is $\mu_X=\sum_{i}\delta_{x_i}$ where $\delta_{x_i}$ is the Dirac measure on $x_i$. For any function from the plane to some Banach space we get a resulting linear representation via integration over these persistence measures.}

\begin{definition}
	Let $\mathcal{B}$ be a Banach space. A \emph{linear representation} is a function $\Phi:\D\to \mathcal{B}$ such that $\Phi(X)=\int_{\R^{2+}}f(x)d\mu_X(x)=\sum_{x_i \in X} f(x_i)$ for some $f:\overline{\R}^{2+}\to \mathcal{B}$.
\end{definition}

As these topological summaries lie in Banach spaces, often even Hilbert spaces, the number of statistical methods available for analysis increases. Often these constructions of linear representations are justified as maintaining relevant persistence homology information because of stability with respect to $1$-Wasserstein distances of the original persistence diagrams.

Lipschitz stability with respect to $1$-Wasserstein distance for persistence diagrams has been shown for a number of linear representations, see for example persistence scale space kernel \cite{reininghaus2015stable} and persistence images \cite{chung2019persistence}. Related theoretic bounds for distances between general linear representations may be found in Divol and Lacombe \cite{divol2020understanding}. \kts{They give necessary and sufficient conditions for when linear representations are continuous with respect to Wasserstein distances. They also show that for a $1$-Lipschitz function $f$ that goes to zero on the diagonal, the corresponding linear representations are $1$-Lipschitz with respect to the $1$-Wasserstein distance. In this section, we complete the story about Lipschitz stability for linear representations into general Banach spaces. Despite the overlap in material with \cite{divol2020understanding}, we include both directions for the sake of completeness and to provide a more elementary proof.} 

%For completeness we recall the necessary and sufficient conditions for Lipschitz stability.
Note that all the $L_q$ metrics over $\R^{2+}\cup \Delta$ are bi-Lipschitz equivalent up to a slight change in constant. \kts{This implies that the choice of $q$ will not affect whether there is Lipschitz stability (though it may affect the Lipschitz constant).} \prsx{The following theorem generalizes and unifies a number of previous results, e.g. \cite{reininghaus2015stable,chung2019persistence}.}  

%For the sake of clarity we will to restrict the case of the $L_1$ metric on $\R^{2+}\cup \Delta$.

%------------------------
%
%In [21] they prove in Proposition 5.2 that for .... for $f$ a function with $f$ zero along the diagonal and Lipschitz with constant less than one that ...
%
%	we are proving the other direction. Also in [21] the result using measure-theoretic techniques so we are including here using only elementary techniques to remain self-contained and more accessible.

%-----------------

\begin{theorem}\label{thm:LipschitzLinear}  Let $\Phi:\D\to \mathcal{B}$ be a non-trivial linear representation constructed via  $f:\overline{\R}^{2+} \to \mathcal{B}$.
		Then $\Phi$ is Lipschitz continuous with respect to $W_p$ with constant $C$ if and only if $p=1$ and $f$ is Lipschitz continuous with constant $C$ and $0$ on the set $\{(t,t) \mid t\in \R\}$.
\end{theorem}

\begin{proof}
	Let us first assume that  $\Phi:\D\to \mathcal{B}$ is Lipschitz continuous with respect to $W_p$ with constant $C$.  \kt{We will show that $p=1$ by way of contradiction.} Let $x\in \R^{2+}$ with $f(x)\neq 0$. Set $X$ to be the persistence diagram consisting of $k$ copies of $x$, and $Y$ the persistence diagram containing no off-diagonal points. Now
	$$W_p(X,Y)=(k \|x-\Delta\|_p^p)^{1/p}=k^{1/p}\|x-\Delta\|_p.$$
	In contrast
	$\|\Phi(X)-\Phi(Y)\|_{\mathcal{B}}=\|\Phi(X)\|_{\mathcal{B}}=\|k \cdot f(x)\|_{\mathcal{B}}=k\| f(x)\|_{\mathcal{B}}.$

	By assumption we have $$k\| f(x)\|_{\mathcal{B}}\leq C k^{1/p}\|x-\Delta\|_p$$ for all $k$ which clearly creates a contradiction if $p>1$.

	From now on we set $p=1$. To show $f$ is Lipschitz, let $x,y\in \R^{2+}$ and set $X$ and $Y$ to be the persistence diagrams no off-diagonal elements except $x$ or $y$ respectively.
	We have
	\begin{align*}
		\|f(x)-f(y)\|_{\mathcal{B}}=\|\Phi(x)-\Phi(y)\|_{\mathcal{B}}\leq CW_1(X,Y)\leq C\|x-y\|_1
	\end{align*}
	where the first inequality follows by assumption and the second because $\phi(x)=y$ determines a matching (which may not necessarily be optimal).   
	
	Now set $X$ to be the persistence diagram with $x\in \R^{2+}$ the only off-diagonal point and $Y$ the persistence diagram containing no off-diagonal points. By definition $\Phi(Y)=0$, and $\Phi(X)=f(x)$. Thus
	$$\|f(x)\|=\|\Phi(X) -\Phi(Y)\|\leq C W_1(X,Y)=C\|x-\Delta\|_1$$
which implies that $f(s,t) \to 0$ as $|s-t|\to 0$. Note that if we have a persistence diagram containing just the diagonal point $(t,t)$ we have $\|\Phi(X)=\|f((t,t)\|\leq CW_1(X,Y)=0$.

	To prove the other direction, suppose $\|f(x)-f(y)\|\leq C\|x-y\|_1$ for all $x,y\in \R^{2+}$ and $f((t,t))=0$ for all $t$. Let $X,Y$ be persistence diagrams and let $\Mch$ be a matching between them.
	\begin{align*}
		\|\Phi(X)-\Phi(Y)\| & =\left\|\sum_{x\in X} f(x)-\sum_{y\in Y} f(y)\right\| \\
		                    & \leq \left\|\sum_{(x,y)\in \Mch}f(x)-f(y)\right\|     \\ &\leq \sum_{(x,y)\in \Mch}\left\|f(x)-f(y)\right\|\\
		                    & \leq \sum_{(x,y)\in \Mch}C \left\|x-y\right\|_1
	\end{align*}
	This holds for all matchings $\Mch$ and hence $\|\Phi(X)-\Phi(Y)\|\leq C W_1(X,Y)$.
\end{proof}

\begin{definition}
	We define the $k$-th dimensional persistent homology rank function corresponding to the filtration $K$ to be
	\begin{align*}
		\beta_k(K) & : \R^{2+} \to \Z                                    \\
		(a,b)      & \mapsto \rk  (\Hg_k(K_a) \rightarrow \Hg_k(K_b))
	\end{align*}
	where $K_a$ is the filtration at $a$.
	We define a \emph{weighting function} as any real valued function $\psi:\R^{2+} \rightarrow \R_{\geq 0}$ such $\psi$ is non-zero on a non-zero measure of $\R^{2+}$.  
	
	We define the $(L^q, \psi)$-\emph{weighted norm} as
	\begin{align}
		||g||_{q, \psi}  =\left(  \int_{x<y}|g|^q \psi(x,y)\, dx\, dy\right)^{\frac{1}{q}}.
	\end{align}

	%viewed as a measure space with density function $\phi(y-x)$. %The $q$-norm for this Banach space is $\|g\|_q^q=\int_{x<y}|g|^q \phi(y-x)\, dx\, dy$. }
\end{definition}
		We can define the Banach space of real valued functions over $\R^{2+}$ using this weighted $q$-norm. Denote this Banach space by $L^q(\R^{2+}, \psi)$. One option of the weight function is $\psi(t) = e^{-t}$ which was used in \cite{robins2016principal}. Here we will completely characterise the weight functions that will allow for Lipschitz stability with respect to Wasserstein distances.

%Although each persistent homology rank function can only have integer values at each point in $\R^{2+}$ we wish to consider the larger space of real valued functions in order to be able to perform statistical methods, such as computing means, or allowing rescaling for normalisation.
Before we prove the theorem characterising when we have Lipschitz stability we first need to recall a standard measure theory result about Lebesgue density.

\begin{lemma}\label{lem:>T}[Corollary 1.5 in Chapter 3 of \cite{stein2009real}]
Fix $\alpha\in (0,1)$. If $A\subset \R$ has positive measure then there exists $a\in A$ such that for all sufficiently small $r$, the measure of $A\cap[a-r, a+r]$ is at least $2r\alpha$.
\end{lemma}

\begin{theorem}\label{cor:rank}
Persistent homology rank functions with the $(L^q, \psi)$ weighted metric are Lipschitz continuous with respect to the $p$-Wasserstein distances between diagrams, with Lipschitz constant $C$,  if and only if $q=p=1$ and $\psi$ satisfies the following:
	\begin{itemize}
	\item	$\int_x^\infty \psi(x,t)\, dt \leq C$ for almost all $x$, and	
	\item $\int_{-\infty}^{y}\psi(t,y)\, dt\leq C$ for almost all $y$. 
	\end{itemize}
	%In this case, the Lipschitz constant is $1$.
\end{theorem}

\begin{proof}
	We can see that the persistent homology rank function is a linear representation of persistence diagrams. Define $f:\R^{2+} \to L^q(\R^{2+}, \psi)$ by $f(a,b)=1_{\{(x,y):a\leq x\leq y\leq b\}}$. Then we can observe that for any diagram $X$ we have $\beta(X)=\sum_{x\in X} f(x)$.

Suppose that the persistent homology rank functions with $(L^q, \psi)$ weighted metric are Lipschitz continuous  with respect to the $p$-Wasserstein distances between diagrams, with Lipschitz constant $C$. Since persistent homology rank functions are linear representations we can apply Theorem \ref{thm:LipschitzLinear} which implies $p=1$. 
	
	Define the function $\rho:\R \to \R\cup \{\infty\}$ by $\rho(x)=\int_x^\infty \psi(x,y)\, dy$. 
	
	%We will next show that we will also need $q=1$ through a counterexample.

	Let $x_1\leq x_2 \leq y_0$ and consider the persistent homology rank functions constructed from the persistence diagrams containing a single off-diagonal point $(x_1, y_0)$ and $(x_2, y_0)$ respectively. Note that for large $y_0$, the $1$-Wasserstein distance between these persistence diagrams is $x_2-x_1$. The $\psi$-weighted $q$-distance between these persistent homology rank functions is
	$$\|f(x_1, y_0)-f(x_2, y_0)\|_{q, \psi}=\left(\int_{x_1}^{x_2} \left(\int_{x}^{y_0} \psi(x,y) \,dy \right)\,dx \right)^{1/q}.$$
	From our Lipschitz assumption we see that 
	$$\left(\int_{x_1}^{x_2} \left(\int_{x}^{y_0} \psi(x,y) \,dy \right)\,dx \right)^{1/q}\leq C|x_2-x_1|.$$
%		$$\|f(x_1, y)-f(x_2, y)\|_q=\left(\int_{x_1}^{x_2} \int_{t}^y e^{t-s}ds\right)^{1/q}\, dt=\left((x_2-x_1)-e^{x_2-y}+e^{x_1-y}\right)^{1/q}.$$
As this holds for all large $y$ we can take the limit on both sides and see that 
$\left(\int_{x_1}^{x_2} \rho(x)\, dx \right)^{1/q}\leq C|x_2-x_1|$
and thus 
\begin{equation}\label{eq:C,q}
\int_{x_1}^{x_2} \rho(x)\, dx\leq C^q|x_2-x_1|^q
\end{equation} for all $x_1\leq x_2$. 
We first will show that this is  only possible if $q=1$. As $\psi$ is a weighing function there exists a threshold $T>0$ such that $A^{T}=\{x\mid \rho(x)>T\}$ has positive measure. Fix $\alpha\in (0,1)$. By Lemma \ref{lem:>T} there exists $a\in \R$  such that for all sufficiently small $r$we have the 
$\int_{a-r}^{a+r} \rho(x)\, dx\geq 2r\alpha$. Combining with \eqref{eq:C,q} we get
$$2r\alpha\leq C^q2^q r^q$$ for all sufficiently small $r$. As $C$ is a constant, this is clearly impossible if $q>1$.

Suppose now that the set of $x$ with $\rho(x)> C$ has positive measure. This implies that there is a threshold $T>C$ such that $A^T=\{x\mid \rho(x)>T\}$ has positive measure. Choose $\alpha<1$ such that $T\alpha>C$. By Lemma \ref{lem:>T} there exists $a$ such that  $$\int_{a-r}^{a+r} \rho(x)\, dx \geq 2r \alpha T > 2rC$$ for all sufficiently small $r$. This contradicts \eqref{eq:C,q} so we can conclude that $\rho(x)\leq C$ for almost all $x$. A symmetric argument reversing the roles of $x$ and $y$ implies that $\int_{-\infty}^{y}\psi(t,y)\, dt\leq C$ for almost all $y$.
	
	%For each fixed $x_1, x_2$ we can consider the limit as $y$ goes to infinity. This limit is
	%$$\lim_{y\to \infty}\|f(x_1, y)-f(x_2, y)\|_q= (x_2-x_1)^{1/q}$$

%	If the persistent homology rank functions with $L^q$ weighted metric are Lipschitz continuous  then there is some $C$ such that for all $x_1\leq x_2<y$ we have $\|f(x_1, y)-f(x_2,y)\|_q\leq C(x_2-x_1)$. Combining with the limit above we have $(x_2-x_1)^{1/q}\leq C(x_2-x_1)$ for all $x_1<x_2$.
%		If $q>1$ there is no constant $C$ that will satisfy this and thus by contradiction we see that is the persistent homology rank functions with $L^q$ weighted metric are Lipschitz continuous with respect to the $p$-Wasserstein distances between diagrams then $q=1$.

To show the other direction, set $p=q=1$ and suppose weighting function $\psi$ satisfies 	
\begin{itemize}
	\item	$\int_x^\infty \psi(x,t)\, dt \leq C$ for almost all $x$, and	
	\item $\int_{-\infty}^{y}\psi(t,y)\, dt\leq C$ for almost all $y$. 
\end{itemize}

By Theorem  \ref{thm:LipschitzLinear} it suffices to show that for $f:\R^{2+} \to L^q(\R^{2+}, \psi)$ by $f(a,b)=1_{\{(x,y):a\leq x\leq y\leq b\}}$ we have $$\|f(x_1, y_1)-f(x_2,y_2)\|_{1,\psi}\leq C(|x_1-x_2| + |y_1-y_2|).$$

	Without loss of generality assume $x_1\leq x_2$. There are three cases to consider 
	\begin{itemize}
	\item[(i)] $x_1\leq y_1\leq x_2\leq y_2$
	\item[(ii)]$x_1\leq x_2\leq y_1\leq y_2$
	\item[(iii)]$x_1\leq x_2\leq y_2 \leq y_1$
	\end{itemize}

In all three cases $$\|f(x_1, y_1)-f(x_2, y_2)\|_{1, \psi}\leq \|1_{\{(x,y)\in \R^{2+}\mid x_1\leq  x \leq x_2\}} + 1_{\{(x,y)\in \R^{2+}\mid \min\{y_1, y_2\}\leq y\leq \max\{ y_1,y_2\}\}}\|_{1,\psi}.$$ This follows from the the containment $\supp(f(x_1, y_1))-f(x_2, y_2))\subset M(x_1,x_2,y_1,y_2)$ where $M(x_1,x_2,y_1,y_2)=\{(x,y)\in \R^{2+}\mid x_1\leq  x \leq x_2\}\cup \{(x,y)\in \R^{2+}\mid \min\{y_1, y_2\}$ which is illustrated in Figure \ref{fig:rankfunctions}.

\begin{figure}[htbp]
\centering

% First row
\begin{subfigure}[b]{0.3\textwidth}
\centering
\begin{tikzpicture}[scale=0.75]
    % Define coordinates for x_1 < y_1 < x_2 < y_2
    \def\x{0.5}
    \def\y{1.5}
    \def\xx{2.5}
    \def\yy{3.5}
    
    % Draw axes
    \draw[->] (-1,0) -- (5.5,0) node[right] {$x$};
    \draw[->] (0,-1) -- (0,5.5) node[above] {$y$};
    
    % Draw the diagonal y = x
    \draw[thick, dashed] (-1,-1) -- (5,5) node[above right] {$y = x$};
    
    % Shade both triangular regions in grey (no overlap since y_1 < x_2)
    % K(x_1, y_1): triangle with vertices (x_1, x_1), (x_1, y_1), (y_1, y_1)
    \fill[gray!50, opacity=0.8] (\x,\x) -- (\x,\y) -- (\y,\y) -- cycle;
    
    % K(x_2, y_2): triangle with vertices (x_2, x_2), (x_2, y_2), (y_2, y_2)
    \fill[gray!50, opacity=0.8] (\xx,\xx) -- (\xx,\yy) -- (\yy,\yy) -- cycle;
    
    % Draw boundaries of both regions
    \draw[black, thick] (\x,\x) -- (\x,\y) -- (\y,\y) -- cycle;
    \draw[black, thick] (\xx,\xx) -- (\xx,\yy) -- (\yy,\yy) -- cycle;
    
    % Mark points on the diagonal (in order x_1, y_1, x_2, y_2)
    \fill (\x,\x) circle (1.5pt);
    \fill (\y,\y) circle (1.5pt);
    \fill (\xx,\xx) circle (1.5pt);
    \fill (\yy,\yy) circle (1.5pt);
    
    % Label points on the diagonal
    \node[below right] at (\x,\x) {$x_1$};
    \node[below right] at (\y,\y) {$y_1$};
    \node[below right] at (\xx,\xx) {$x_2$};
    \node[below right] at (\yy,\yy) {$y_2$};
    
\end{tikzpicture}

\caption{$\supp(f(x_1, y_1))-f(x_2, y_2))$ for $x_1\leq y_1\leq x_2\leq y_2$ (case (i))}
\label{fig:sub1}
\end{subfigure}
\hfill
\begin{subfigure}[b]{0.3\textwidth}
\centering

\begin{tikzpicture}[scale=0.75]
    % Define coordinates
    \def\x{0.5}
    \def\xx{1.5}
    \def\y{2.5}
    \def\yy{3.5}
    
    % Draw axes
    \draw[->] (-1,0) -- (5.5,0) node[right] {$x$};
    \draw[->] (0,-1) -- (0,5.5) node[above] {$y$};
    
    % Draw the diagonal y = x
    \draw[thick, dashed] (-1,-1) -- (5,5) node[above right] {$y = x$};
    
    % Shade the symmetric difference parts in grey
    % Part 1: K(x_1, y_1) \ K(x_2, y_1) = region where x_1 ≤ x < x_2, x ≤ y ≤ y_1
    \fill[gray!50, opacity=0.8] (\x,\x) -- (\x,\y) -- (\xx,\y) -- (\xx,\xx) -- cycle;
    
    % Part 2: K(x_2, y_2) \ K(x_2, y_1) = region where x_2 ≤ x ≤ y, y_1 < y ≤ y_2
    \fill[gray!50, opacity=0.8] (\xx,\y) -- (\xx,\yy) -- (\yy,\yy) -- (\y,\y) -- cycle;
    
    % Draw boundaries of all regions
    \draw[black, thick] (\x,\x) -- (\x,\y) -- (\y,\y) -- cycle;
    \draw[black, thick] (\xx,\xx) -- (\xx,\yy) -- (\yy,\yy) -- cycle;
    \draw[black, thick] (\xx,\xx) -- (\xx,\y) -- (\y,\y) -- cycle;
    
    % Mark points on the diagonal
    \fill (\x,\x) circle (1.5pt);
    \fill (\xx,\xx) circle (1.5pt);
    \fill (\y,\y) circle (1.5pt);
    \fill (\yy,\yy) circle (1.5pt);
    
    % Label points on the diagonal
    \node[below right] at (\x,\x) {$x_1$};
    \node[below right] at (\xx,\xx) {$x_2$};
    \node[below right] at (\y,\y) {$y_1$};
    \node[below right] at (\yy,\yy) {$y_2$};
    
\end{tikzpicture}

\caption{$\supp(f(x_1, y_1))-f(x_2, y_2))$ for $x_1\leq x_2\leq y_1\leq y_2$ (case (ii))}
\label{fig:sub2}
\end{subfigure}
\hfill
\begin{subfigure}[b]{0.3\textwidth}
\centering

\begin{tikzpicture}[scale=0.75]
    % Define coordinates for x_1 < x_2 < y_2 < y_1
    \def\x{0.5}
    \def\xx{1.5}
    \def\yy{2.5}
    \def\y{3.5}
    
    % Draw axes
    \draw[->] (-1,0) -- (5.5,0) node[right] {$x$};
    \draw[->] (0,-1) -- (0,5.5) node[above] {$y$};
    
    % Draw the diagonal y = x
    \draw[thick, dashed] (-1,-1) -- (5,5) node[above right] {$y = x$};
    
    % Shade the symmetric difference K(x_1, y_1) \ K(x_2, y_2) in grey
    % Since K(x_2, y_2) ⊆ K(x_1, y_1), the symmetric difference is K(x_1, y_1) \ K(x_2, y_2)
    
    % Left piece: x_1 ≤ x < x_2, x ≤ y ≤ y_1
    \fill[gray!50, opacity=0.8] (\x,\x) -- (\x,\y) -- (\xx,\y) -- (\xx,\xx) -- cycle;
    
    % Upper piece: x_2 ≤ x ≤ y, y_2 < y ≤ y_1  
    \fill[gray!50, opacity=0.8] (\xx,\yy) -- (\xx,\y) -- (\y,\y) -- (\yy,\yy) -- cycle;
    
    % Draw boundaries
    \draw[black, thick] (\x,\x) -- (\x,\y) -- (\y,\y) -- cycle; % K(x_1, y_1)
    \draw[black, thick] (\xx,\xx) -- (\xx,\yy) -- (\yy,\yy) -- cycle; % K(x_2, y_2)
    
    % Mark points on the diagonal (in order x_1, x_2, y_2, y_1)
    \fill (\x,\x) circle (1.5pt);
    \fill (\xx,\xx) circle (1.5pt);
    \fill (\yy,\yy) circle (1.5pt);
    \fill (\y,\y) circle (1.5pt);
    
    % Label points on the diagonal
    \node[below right] at (\x,\x) {$x_1$};
    \node[below right] at (\xx,\xx) {$x_2$};
    \node[below right] at (\yy,\yy) {$y_2$};
    \node[below right] at (\y,\y) {$y_1$};
    
\end{tikzpicture}

\caption{$\supp(f(x_1, y_1))-f(x_2, y_2))$ for $x_1\leq x_2\leq y_2\leq y_1$ (case (iii))}
\label{fig:sub3}
\end{subfigure}

\vspace{1cm}

% Second row
\begin{subfigure}[b]{0.3\textwidth}
\centering

\begin{tikzpicture}[scale=0.75]
    % Define coordinates for x_1 < y_1 < x_2 < y_2
    \def\x{0.5}
    \def\y{1.5}
    \def\xx{2.5}
    \def\yy{3.5}
    
    % Draw axes
    \draw[->] (-1,0) -- (5.5,0) node[right] {$x$};
    \draw[->] (0,-1) -- (0,5.5) node[above] {$y$};
    
    % Draw the diagonal y = x
    \draw[thick, dashed] (-1,-1) -- (5,5) node[above right] {$y = x$};
    
    % Shade M(y_1, y_2): region where x ≤ y and y_1 ≤ y ≤ y_2
    % This is the area above the diagonal between y = y_1 and y = y_2
    \fill[gray!50, opacity=0.8] (-1,\y) -- (\y,\y) -- (\yy,\yy) -- (-1,\yy) -- cycle;
    
    % Shade L(x_1, x_2): region where x ≤ y and x_1 ≤ x ≤ x_2
    % This is the area above the diagonal between x = x_1 and x = x_2
    \fill[gray!50, opacity=0.8] (\x,\x) -- (\x,5) -- (\xx,5) -- (\xx,\xx) -- cycle;
    
    % Draw boundaries
    \draw[black, thick] (-1,\y) -- (\y,\y) -- (\yy,\yy) -- (-1,\yy);
    \draw[black, thick] (\x,\x) -- (\x,5);
    \draw[black, thick] (\xx,\xx) -- (\xx,5);
    \draw[black, thick] (\x,\x) -- (\xx,\xx);
    
    % Mark points on the diagonal (in order x_1, y_1, x_2, y_2)
    \fill (\x,\x) circle (1.5pt);
    \fill (\y,\y) circle (1.5pt);
    \fill (\xx,\xx) circle (1.5pt);
    \fill (\yy,\yy) circle (1.5pt);
    
    % Label points on the diagonal
    \node[below right] at (\x,\x) {$x_1$};
    \node[below right] at (\y,\y) {$y_1$};
    \node[below right] at (\xx,\xx) {$x_2$};
    \node[below right] at (\yy,\yy) {$y_2$};

\end{tikzpicture}

\caption{$M(x_1,x_2,y_1,y_2)$ in case (i)}
\label{fig:sub4}
\end{subfigure}
\hfill
\begin{subfigure}[b]{0.3\textwidth}
\centering

\begin{tikzpicture}[scale=0.75]
    % Define coordinates for x_1 < x_2 < y_1 < y_2
    \def\x{0.5}
    \def\xx{1.5}
    \def\y{2.5}
    \def\yy{3.5}
    
    % Draw axes
    \draw[->] (-1,0) -- (5.5,0) node[right] {$x$};
    \draw[->] (0,-1) -- (0,5.5) node[above] {$y$};
    
    % Draw the diagonal y = x
    \draw[thick, dashed] (-1,-1) -- (5,5) node[above right] {$y = x$};
    
    % Shade M(y_1, y_2): region where x ≤ y and y_1 ≤ y ≤ y_2
    % This is the area above the diagonal between y = y_1 and y = y_2
    \fill[gray!50, opacity=0.8] (-1,\y) -- (\y,\y) -- (\yy,\yy) -- (-1,\yy) -- cycle;
    
    % Shade L(x_1, x_2): region where x ≤ y and x_1 ≤ x ≤ x_2
    % This is the area above the diagonal between x = x_1 and x = x_2
    \fill[gray!50, opacity=0.8] (\x,\x) -- (\x,5) -- (\xx,5) -- (\xx,\xx) -- cycle;
    
    % Draw boundaries
    \draw[black, thick] (-1,\y) -- (\y,\y) -- (\yy,\yy) -- (-1,\yy);
    \draw[black, thick] (\x,\x) -- (\x,5);
    \draw[black, thick] (\xx,\xx) -- (\xx,5);
    \draw[black, thick] (\x,\x) -- (\xx,\xx);
    
    % Mark points on the diagonal
    \fill (\x,\x) circle (1.5pt);
    \fill (\xx,\xx) circle (1.5pt);
    \fill (\y,\y) circle (1.5pt);
    \fill (\yy,\yy) circle (1.5pt);
    
    % Label points on the diagonal
    \node[below right] at (\x,\x) {$x_1$};
    \node[below right] at (\xx,\xx) {$x_2$};
    \node[below right] at (\y,\y) {$y_1$};
    \node[below right] at (\yy,\yy) {$y_2$};

\end{tikzpicture}

\caption{$M(x_1,x_2,y_1,y_2)$ in case (ii)}
\label{fig:sub5}
\end{subfigure}
\hfill
\begin{subfigure}[b]{0.3\textwidth}
\centering
\begin{tikzpicture}[scale=0.75]
    % Define coordinates for x_1 < x_2 < y_1 < y_2
    \def\x{0.5}
    \def\xx{1.5}
    \def\y{2.5}
    \def\yy{3.5}
    
    % Draw axes
    \draw[->] (-1,0) -- (5.5,0) node[right] {$x$};
    \draw[->] (0,-1) -- (0,5.5) node[above] {$y$};
    
    % Draw the diagonal y = x
    \draw[thick, dashed] (-1,-1) -- (5,5) node[above right] {$y = x$};
    
    % Shade M(y_1, y_2): region where x ≤ y and y_1 ≤ y ≤ y_2
    % This is the area above the diagonal between y = y_1 and y = y_2
    \fill[gray!50, opacity=0.8] (-1,\y) -- (\y,\y) -- (\yy,\yy) -- (-1,\yy) -- cycle;
    
    % Shade L(x_1, x_2): region where x ≤ y and x_1 ≤ x ≤ x_2
    % This is the area above the diagonal between x = x_1 and x = x_2
    \fill[gray!50, opacity=0.8] (\x,\x) -- (\x,5) -- (\xx,5) -- (\xx,\xx) -- cycle;
    
    % Draw boundaries
    \draw[black, thick] (-1,\y) -- (\y,\y) -- (\yy,\yy) -- (-1,\yy);
    \draw[black, thick] (\x,\x) -- (\x,5);
    \draw[black, thick] (\xx,\xx) -- (\xx,5);
    \draw[black, thick] (\x,\x) -- (\xx,\xx);
    
    % Mark points on the diagonal
    \fill (\x,\x) circle (1.5pt);
    \fill (\xx,\xx) circle (1.5pt);
    \fill (\y,\y) circle (1.5pt);
    \fill (\yy,\yy) circle (1.5pt);
    
    % Label points on the diagonal
    \node[below right] at (\x,\x) {$x_1$};
    \node[below right] at (\xx,\xx) {$x_2$};
    \node[below right] at (\y,\y) {$y_2$};
    \node[below right] at (\yy,\yy) {$y_1$};

\end{tikzpicture}

\caption{$M(x_1,x_2,y_1,y_2)$ in case (iii)}
\label{fig:sub6}
\end{subfigure}

\caption{Illustration showing $\supp(f(x_1, y_1))-f(x_2, y_2))\subset M(x_1,x_2,y_1,y_2)$ where $M(x_1,x_2,y_1,y_2)=\{(x,y)\in \R^{2+}\mid x_1\leq  x \leq x_2\}\cup \{(x,y)\in \R^{2+}\mid \min\{y_1, y_2\}$ in all three cases.}
\label{fig:rankfunctions}
\end{figure}

We thus can use this to bound $\|f(x_1, y_1)-f(x_2, y_2)\|_{1, \psi}$ with 
\begin{align*}
\|f(x_1, y_1)-f(x_2, y_2)\|_{1, \psi}%&=\|1_{\text{support}(f(x_1, y_1)-f(x_2, y_2))}\|_{1,\psi}\\
&\leq \int_{x_1}^{x_2}\left(\int_x^\infty \psi(x,y) \, dy\right) \, dx + \int_{\min\{y_1, y_2\}}^{\max\{y_1, y_2\}}\left(\int_{-\infty}^y \psi(x,y) \, dx\right) \, dy\\
&\leq C|x_1-x_2|+C|y_1-y_2|.
\end{align*}
%	
%	 If $x_2\leq y_1$ then
%	$$\|f(x_1, y_1)-f(x_2, y_2)\|_1\leq \|f(x_1, y_1)-f(x_2, y_1)\|_1+ \|f(x_2, y_1)-f(x_2, y_2)\|_1.$$
%	Using the integral above we see that $\|f(x_1, y_1)-f(x_2, y_1)\|_1\leq |x_2-x_1|$ and analogously that $\|f(x_2, y_1)-f(x_2, y_2)\|_1\leq |y_2-y_1|$. Together they imply that $\|f(x_1, y_1)-f(x_2, y_2)\|_1\leq \|(x_1, y_1)-(x_2, y_2)\|_1$.
%
%	If $x_2 > y_1$ then the supports of $f(x_1, y_1)$ and $f(x_2, y_2)$ are disjoint. Routine calculations show that $\|f(x, y)\|_1\leq |y-x|$. In this scenario, $\|(x_1,y_1)-(x_2,y_2)\|_1\geq |y_1-x_1|+|y_2-x_2|$ and hence $\|f(x_1, y_1)-f(x_2, y_2)\|_1\leq \|(x_1, y_1)-(x_2, y_2)\|_1$.
\end{proof}

\subsection{Persistence Landscapes are not Lipschitz stable}\label{sec:landscape}

Persistence landscapes~\cite{bubenik2015statistical} were among the first functionals proposed for persistence diagrams and remain among the most popular in practice.
\begin{definition}
	The \emph{persistence landscape} of persistence module $M$ is the function $\lambda:\mathbb{N}\times\R \to \R$ defined by
	$$\lambda(k, t)(M)=\sup\{h\geq 0\mid \rk(M(t-h \leq t+h))\geq k).$$
	We call $\lambda(k, \cdot)(M)$ the $k$-th persistence landscape.
\end{definition}

The $L^q$ distance between persistence landscapes is defined as the sum over $k$ of the $L^q$ distances between the $k$-th persistence landscape.
Let $pl_k(f)$ denote the $k$-th persistence landscape of the sublevel persistence diagram for $f$.

There is a form of bottleneck stability for persistence landscapes, see \cite{bubenik2015statistical},
but unlike the linear functionals in the previous subsection, there is no Lipschitz nor even H\"older stability with respect to the $p$-Wasserstein distances ($p<\infty$) of their corresponding persistence diagrams.

\begin{theorem}\label{thm:pl}
	Let $(\D, W_p)$ denote the space of persistence diagrams with the $W_p$ metric and let $(PL, L_q)$ denote the space of persistence landscapes with the $L^q$ metric. For all $q\in [1, \infty)$, the function $pl: (\D, W_p) \to (PL, L^q)$ which sends each persistence diagram to its corresponding persistence landscape is \emph{not} H\"older continuous.
\end{theorem}

\begin{proof}
	Let $X$ and $Y$ be the persistence diagrams with one off-diagonal point at $(0, a)$ and $(0, a-r)$ respectively, where $r\ll a$.
	The first persistence landscapes for $pl(X)$ and $pl(Y)$  are both a triangle function. These are centred at $a/2$ and $(a-r)/2$ respectively.
	We can compute that $pl(X)-pl(Y)$ is a trapezium shape:
	$$(pl(X)-pl(Y))(t)= \begin{cases}
			t   & \text{for } 2t\in[(a-r)/2, a/2] \\
			r   & \text{for } t\in [a/2, a-r]     \\
			a-t & \text{for } t\in[a-r, a]        \\
			0   & \text{otherwise.}
		\end{cases}$$

	When $a\gg r$, the contribution of the integral over $[a/2, a-r]$ will dominate the $L^q$ distance between $pl(X)$ and $pl(Y)$.
	The function distance is bounded below by
	$$\|pl(X)-pl(Y)\|_q  > \left(\int_{a/2}^{a-r} r^q \, dt\right)^{1/q}  =r(a/2-r)^{1/q}$$

	We also know that for $r\ll a$ the optimal matching between $X$ and $Y$ sends the point at $(0,a)$ to $(0, a-r)$ and hence $W_p(X,Y)=r$ for all $p\in [1,\infty]$.
	For a H\"older stability result to hold we would need there to be $\alpha, C>0$ such that
	$\|pl(X)-pl(Y)\|_q\leq CW_p(X,Y)^\alpha$ for all $X,Y\in \D$.

	This would imply
	$$r(a/2-r)^{1/q}\leq Cr^\alpha \qquad \text{for all } a\gg r.$$

	By setting $r$ small and $a$ large we can make the left hand side arbitrarily large and the right hand side arbitrarily small which provides a contradiction regardless of the choice of $q, C$ and $\alpha$. This means there cannot be any H\"older continuity when $q\neq \infty$.
\end{proof}

\begin{corollary}
	Let $M$ be a simplicial complex containing at least one edge. Let $(X, L^p)$ denote the space of monotone functions over $M$ with the $L^p$ metric. For all $p,q\in [1, \infty)$, the function $PL: (X, d_{L^p}) \to (pl,L^q)$ which sends each function to the persistence landscape of its sublevel set filtration is \emph{not} H\"older continuous.
\end{corollary}

\begin{proof}
	We prove the result by creating an example of a pair of functions that produce the persistence diagrams in Theorem \ref{thm:pl} as a \prsx{subdiagram.} Fix an edge $[x_1, x_2]$ in $M$. Set $f([x_1])=0$, $f([x_2])=0$, $g([x_1,x_2])=a-r$ and $g(\tau)=a$ for all other cells $\tau\in M$. 	\prsx{Let all other simplices take a constant value larger than $2a$ for both $f$ and $g$.} Note that $\|f-g\|_p=r$ for all $r\in [0,a]$. 
	The persistence diagram of the sublevel set filtrations of $f$ and $g$ \prsx{until $a$} are the $X$ and $Y$ used in the proof of Theorem \ref{thm:pl} \prsx{and the same for both diagrams outside of this subdiagram. More precisely,  they contain the essential classes of $M$ which appear at the chosen constant.} The remainder of the proof are the same inequalities as before.
\end{proof}

\begin{remark} It is worth noting that \cite{bubenik2015statistical} does have a limited version of Wasserstein stability using \cite{cohen2010lipschitz}. This corollary states that for $X$ a triangulable, compact metric space that implies bounded degree-$k$ total persistence for some real number $k \geq 1$, and $f, g$ two tame Lipschitz functions we have
	$$\left\|PL(f)- PL(g)\right\|_p \leq C\left\| f-g\right\|_\infty^{\frac{p-k}{p}}$$
	for all $p \geq k$, where
	$$C = C_{X,k}\|f\|_\infty (Lip(f)^k + Lip(g)^k) + C_{X,k+1} \frac{1}{p+1} (Lip(f)^{k+1} + Lip(g)^{k+1}).$$
	See Section \ref{sec:existing_limitations} for some limitations in terms of $k$ and $C_{X,k}$.
\end{remark}

\section{Funding}
KT is the recipient of an Australian Research Council Discovery Early Career Award (project number DE200100056) funded by the Australian Government.

\section{Declarations}
\subsection{Conflicts of Interest}
On behalf of all authors, the corresponding author states that there is no conflict of interest.
\bibliographystyle{plain}
\bibliography{wasserstein}

\begin{thebibliography}{10}

\bibitem{adams2017persistence}
Henry Adams, Tegan Emerson, Michael Kirby, Rachel Neville, Chris Peterson,
  Patrick Shipman, Sofya Chepushtanova, Eric Hanson, Francis Motta, and Lori
  Ziegelmeier.
\newblock Persistence images: A stable vector representation of persistent
  homology.
\newblock {\em The Journal of Machine Learning Research}, 18(1):218--252, 2017.

\bibitem{arnal2024wasserstein}
Charles Arnal, David Cohen-Steiner, and Vincent Divol.
\newblock Wasserstein convergence of {\v{c}}ech persistence diagrams for
  samplings of submanifolds.
\newblock 2024.

\bibitem{bauer2014induced}
Ulrich Bauer and Michael Lesnick.
\newblock Induced matchings of barcodes and the algebraic stability of
  persistence.
\newblock In {\em Proceedings 30th Annual Symposium on Computational Geometry},
  page 355. ACM, 2014.

\bibitem{bleile2022persistent}
Bea Bleile, Ad{\'e}lie Garin, Teresa Heiss, Kelly Maggs, and Vanessa Robins.
\newblock The persistent homology of dual digital image constructions.
\newblock In {\em Research in Computational Topology 2}, pages 1--26. Springer,
  2022.

\bibitem{bobrowski2019random}
Omer Bobrowski and Goncalo Oliveira.
\newblock Random {\v{c}}ech complexes on riemannian manifolds.
\newblock {\em Random Structures \& Algorithms}, 54(3):373--412, 2019.

\bibitem{bubenik2015statistical}
Peter Bubenik.
\newblock Statistical topological data analysis using persistence landscapes.
\newblock {\em The Journal of Machine Learning Research}, 16(1):77--102, 2015.

\bibitem{bubenik2015metrics}
Peter Bubenik, Vin de~Silva, and Jonathan Scott.
\newblock Metrics for generalized persistence modules.
\newblock {\em Foundations of Computational Mathematics}, 15(6):1501--1531,
  2015.

\bibitem{Bubenik_2014}
Peter Bubenik and Jonathan~A. Scott.
\newblock Categorification of persistent homology.
\newblock {\em Discrete and Computational Geometry}, 51(3):600--627, Jan 2014.

\bibitem{carriere2017sliced}
Mathieu Carriere, Marco Cuturi, and Steve Oudot.
\newblock Sliced wasserstein kernel for persistence diagrams.
\newblock In {\em Proceedings of the 34th International Conference on Machine
  Learning-Volume 70}, pages 664--673. JMLR. org, 2017.

\bibitem{chazal2012structure}
Fr{\'e}d{\'e}ric Chazal, Vin De~Silva, Marc Glisse, and Steve Oudot.
\newblock The structure and stability of persistence modules.
\newblock {\em arXiv preprint arXiv:1207.3674}, 2012.

\bibitem{chazal2014persistence}
Fr{\'e}d{\'e}ric Chazal, Vin De~Silva, and Steve Oudot.
\newblock Persistence stability for geometric complexes.
\newblock {\em Geometriae Dedicata}, 173(1):193--214, 2014.

\bibitem{chen2011diffusion}
Chao Chen and Herbert Edelsbrunner.
\newblock Diffusion runs low on persistence fast.
\newblock In {\em 2011 International Conference on Computer Vision}, pages
  423--430. IEEE, 2011.

\bibitem{chen2019topological}
Chao Chen, Xiuyan Ni, Qinxun Bai, and Yusu Wang.
\newblock A topological regularizer for classifiers via persistent homology.
\newblock In {\em The 22nd International Conference on Artificial Intelligence
  and Statistics}, pages 2573--2582. PMLR, 2019.

\bibitem{chung2024altered}
Moo~K Chung, Tahmineh Azizi, Jamie~L Hanson, Andrew~L Alexander, Seth~D Pollak,
  and Richard~J Davidson.
\newblock Altered topological structure of the brain white matter in maltreated
  children through topological data analysis.
\newblock {\em Network Neuroscience}, 8(1):355--376, 2024.

\bibitem{chung2019persistence}
Yu-Min Chung and Austin Lawson.
\newblock Persistence curves: A canonical framework for summarizing persistence
  diagrams.
\newblock {\em arXiv preprint arXiv:1904.07768}, 2019.

\bibitem{CohEdeHar2007}
David Cohen-Steiner, Herbert Edelsbrunner, and John Harer.
\newblock Stability of persistence diagrams.
\newblock {\em Discrete and Computational Geometry}, 37:103--120, 2007.

\bibitem{cohen2010lipschitz}
David Cohen-Steiner, Herbert Edelsbrunner, John Harer, and Yuriy Mileyko.
\newblock Lipschitz functions have $l_p$-stable persistence.
\newblock {\em Foundations of computational mathematics}, 10(2):127--139, 2010.

\bibitem{cohen2006vines}
David Cohen-Steiner, Herbert Edelsbrunner, and Dmitriy Morozov.
\newblock Vines and vineyards by updating persistence in linear time.
\newblock In {\em Proceedings of the twenty-second annual symposium on
  Computational geometry}, pages 119--126, 2006.

\bibitem{Crawley_Boevey_2015}
William Crawley-Boevey.
\newblock Decomposition of pointwise finite-dimensional persistence modules.
\newblock {\em Journal of Algebra and Its Applications}, 14(05):1550066, Mar
  2015.

\bibitem{curry2018many}
Justin Curry, Sayan Mukherjee, and Katharine Turner.
\newblock How many directions determine a shape and other sufficiency results
  for two topological transforms.
\newblock {\em arXiv preprint arXiv:1805.09782}, 2018.

\bibitem{delfinado1993incremental}
Cecil Jose~A Delfinado and Herbert Edelsbrunner.
\newblock An incremental algorithm for betti numbers of simplicial complexes.
\newblock In {\em Proceedings of the ninth annual symposium on Computational
  geometry}, pages 232--239, 1993.

\bibitem{divol2020understanding}
Vincent Divol and Th{\'e}o Lacombe.
\newblock Understanding the topology and the geometry of the space of
  persistence diagrams via optimal partial transport.
\newblock {\em Journal of Applied and Computational Topology}, pages 1--53,
  2020.

\bibitem{dunaeva2016classification}
Olga Dunaeva, Herbert Edelsbrunner, Anton Lukyanov, Michael Machin, Daria
  Malkova, Roman Kuvaev, and Sergey Kashin.
\newblock The classification of endoscopy images with persistent homology.
\newblock {\em Pattern Recognition Letters}, 83:13--22, 2016.

\bibitem{EdeHar2010}
Herbert Edelsbrunner and John Harer.
\newblock {\em Computational Topology}.
\newblock American Mathematical Society, 2010.

\bibitem{edelsbrunner2000topological}
Herbert Edelsbrunner, David Letscher, and Afra Zomorodian.
\newblock Topological persistence and simplification.
\newblock In {\em Proceedings 41st annual symposium on foundations of computer
  science}, pages 454--463. IEEE, 2000.

\bibitem{gameiro2016continuation}
Marcio Gameiro, Yasuaki Hiraoka, and Ippei Obayashi.
\newblock Continuation of point clouds via persistence diagrams.
\newblock {\em Physica D: Nonlinear Phenomena}, 334:118--132, 2016.

\bibitem{ghrist2018persistent}
Robert Ghrist, Rachel Levanger, and Huy Mai.
\newblock Persistent homology and euler integral transforms.
\newblock {\em Journal of Applied and Computational Topology}, 2(1-2):55--60,
  2018.

\bibitem{hofer2017deep}
Christoph Hofer, Roland Kwitt, Marc Niethammer, and Andreas Uhl.
\newblock Deep learning with topological signatures.
\newblock In {\em Advances in Neural Information Processing Systems}, pages
  1634--1644, 2017.

\bibitem{hofer2019learning}
Christoph~D Hofer, Roland Kwitt, and Marc Niethammer.
\newblock Learning representations of persistence barcodes.
\newblock {\em Journal of Machine Learning Research}, 20(126):1--45, 2019.

\bibitem{klingenberg1995riemannian}
Wilhelm Klingenberg.
\newblock {\em Riemannian geometry}, volume~1.
\newblock Walter de Gruyter, 1995.

\bibitem{kusano2016persistence}
Genki Kusano, Yasuaki Hiraoka, and Kenji Fukumizu.
\newblock Persistence weighted gaussian kernel for topological data analysis.
\newblock In {\em International Conference on Machine Learning}, pages
  2004--2013, 2016.

\bibitem{leygonie2019framework}
Jacob Leygonie, Steve Oudot, and Ulrike Tillmann.
\newblock A framework for differential calculus on persistence barcodes.
\newblock {\em arXiv preprint arXiv:1910.00960}, 2019.

\bibitem{nauleau2022topological}
Florent Nauleau, Fabien Vivodtzev, Thibault Bridel-Bertomeu, Heloise
  Beaugendre, and Julien Tierny.
\newblock Topological analysis of ensembles of hydrodynamic turbulent flows an
  experimental study.
\newblock In {\em 2022 IEEE 12th Symposium on Large Data Analysis and
  Visualization (LDAV)}, pages 1--11. IEEE, 2022.

\bibitem{peyre2019computational}
Gabriel Peyr{\'e}, Marco Cuturi, et~al.
\newblock Computational optimal transport: With applications to data science.
\newblock {\em Foundations and Trends{\textregistered} in Machine Learning},
  11(5-6):355--607, 2019.

\bibitem{poulenard2018topological}
Adrien Poulenard, Primoz Skraba, and Maks Ovsjanikov.
\newblock Topological function optimization for continuous shape matching.
\newblock In {\em Computer Graphics Forum}, volume~37, pages 13--25. Wiley
  Online Library, 2018.

\bibitem{reininghaus2015stable}
Jan Reininghaus, Stefan Huber, Ulrich Bauer, and Roland Kwitt.
\newblock A stable multi-scale kernel for topological machine learning.
\newblock In {\em Proceedings of the IEEE conference on computer vision and
  pattern recognition}, pages 4741--4748, 2015.

\bibitem{robins2006betti}
Vanessa Robins.
\newblock Betti number signatures of homogeneous poisson point processes.
\newblock {\em Physical Review E—Statistical, Nonlinear, and Soft Matter
  Physics}, 74(6):061107, 2006.

\bibitem{robins2016principal}
Vanessa Robins and Katharine Turner.
\newblock Principal component analysis of persistent homology rank functions
  with case studies of spatial point patterns, sphere packing and colloids.
\newblock {\em Physica D: Nonlinear Phenomena}, 334:99--117, 2016.

\bibitem{robins2011theory}
Vanessa Robins, Peter~John Wood, and Adrian~P Sheppard.
\newblock Theory and algorithms for constructing discrete morse complexes from
  grayscale digital images.
\newblock {\em IEEE Transactions on pattern analysis and machine intelligence},
  33(8):1646--1658, 2011.

\bibitem{saadat2021topological}
Ameer Saadat-Yazdi, Rayna Andreeva, and Rik Sarkar.
\newblock Topological detection of alzheimer’s disease using betti curves.
\newblock In {\em Interpretability of Machine Intelligence in Medical Image
  Computing, and Topological Data Analysis and Its Applications for Medical
  Data: 4th International Workshop, iMIMIC 2021, and 1st International
  Workshop, TDA4MedicalData 2021, Held in Conjunction with MICCAI 2021,
  Strasbourg, France, September 27, 2021, Proceedings 4}, pages 119--128.
  Springer, 2021.

\bibitem{schrijver2003combinatorial}
Alexander Schrijver et~al.
\newblock {\em Combinatorial optimization: polyhedra and efficiency},
  volume~24.
\newblock Springer, 2003.

\bibitem{7789621}
Lee~M. Seversky, Shelby Davis, and Matthew Berger.
\newblock On time-series topological data analysis: New data and opportunities.
\newblock In {\em 2016 IEEE Conference on Computer Vision and Pattern
  Recognition Workshops (CVPRW)}, pages 1014--1022, 2016.

\bibitem{skraba2017randomly}
Primoz Skraba, Gugan Thoppe, and D~Yogeshwaran.
\newblock Randomly weighted $ d-$ complexes: Minimal spanning acycles and
  persistence diagrams.
\newblock {\em arXiv preprint arXiv:1701.00239}, 2017.

\bibitem{stein2009real}
Elias~M Stein and Rami Shakarchi.
\newblock {\em Real analysis: measure theory, integration, and Hilbert spaces}.
\newblock Princeton University Press, 2009.

\bibitem{turner2020medians}
Katharine Turner.
\newblock Medians of populations of persistence diagrams.
\newblock {\em Homology, Homotopy and Applications}, 22(1):255--282, 2020.

\bibitem{turner2014frechet}
Katharine Turner, Yuriy Mileyko, Sayan Mukherjee, and John Harer.
\newblock Fr{\'e}chet means for distributions of persistence diagrams.
\newblock {\em Discrete \&amp; Computational Geometry}, 52(1):44--70, 2014.

\bibitem{turner2014persistent}
Katharine Turner, Sayan Mukherjee, and Doug~M Boyer.
\newblock Persistent homology transform for modeling shapes and surfaces.
\newblock {\em Information and Inference: A Journal of the IMA}, 3(4):310--344,
  2014.

\bibitem{tymochko2020using}
Sarah Tymochko, Elizabeth Munch, Jason Dunion, Kristen Corbosiero, and Ryan
  Torn.
\newblock Using persistent homology to quantify a diurnal cycle in hurricanes.
\newblock {\em Pattern Recognition Letters}, 133:137--143, 2020.

\bibitem{vidal2019progressive}
Jules Vidal, Joseph Budin, and Julien Tierny.
\newblock Progressive wasserstein barycenters of persistence diagrams.
\newblock {\em IEEE transactions on visualization and computer graphics},
  26(1):151--161, 2019.

\bibitem{zomorodian2005computing}
Afra Zomorodian and Gunnar Carlsson.
\newblock Computing persistent homology.
\newblock {\em Discrete \&amp; Computational Geometry}, 33(2):249--274, 2005.

\end{thebibliography}
\appendix

\section{Proof for Section~\ref{sec:preliminaries}}\label{sec:powers}
\begin{lemma}[Lemma \ref{lem:powers}]
	Let $X,Y$ be persistence diagrams such that $\sum_{x\in X}\|x-\Delta\|<\infty$ and $\sum_{y\in Y}\|(y)-\Delta\|<\infty$.  For any  $1\leq \power'< \power$,   $\Was{\power}(X,Y)\leq \Was{\power'}(X,Y)$.
\end{lemma}
\begin{proof}
	Let us first prove the useful inequality that for $t_1, \ldots t_k \geq 0$ and $0<a < 1$ we have
	\begin{equation}\label{eq:useful}
		(t_1+t_2+\ldots t_k)^a \leq t_1^a +t_2^a+\ldots t_k^a.
	\end{equation}
	This inequality can be proved by simple induction on the number of summands. The base case is trivial.
	For the induction step consider the function $f_a(t)=1 + t^a -(1+t)^a$ (with $t\geq 0$) and observe that $f(0)$.
	%To show $f(t)\geq 0$ for $t>0$ we need only show that $f_a'(t)>0$ for $t>0$. Combined with 
	As $f_a'(t)=a(t^{a-1} - (1+t))^{a-1}\geq 0$ we infer that $f(t)\geq 0$. To prove the inductive step we can assume $t_{k+1}>0$ (as otherwise it is trivial) and we calculate
	\begin{align*}
		\frac{(t_1+t_2+\ldots t_k+t_{k+1})^a}{t_{k+1}^a} & = \left(\frac{(t_1+t_2+\ldots t_k)}{t_{k+1}}+1\right)^a    \\
		                                                 & \leq \frac{(t_1+t_2+\ldots t_k+t_{k+1})^a}{t_{k+1}^a}  + 1 \\
		                                                 & \leq  \frac{t_1^a+t_2^a+\ldots t_k^a}{t_{k+1}^a}  + 1.
	\end{align*}

	To use \eqref{eq:useful} we will need to restrict to persistence diagrams constructed from finitely many off-diagonal points. To this end let $0<\epsilon<1$ and choose $X^\epsilon\subset X$ be a persistence diagrams containing finitely many off-diagonal points such that $\sum_{x\in X\backslash X^\epsilon}\|x-\Delta\|<\epsilon$. This is always possible by our assumption that $\sum_{x\in X}\|x-\Delta\|<\infty$.  Similarly construct $Y^\epsilon$. By the triangle inequality we have  $|W_q(X,Y)-W_q(X^\epsilon, Y^\epsilon)|\leq 2\epsilon$ for all $q$.

	Let us now fix a matching $\Mch \subset \{X^\epsilon\cup \Delta\}\times\{ Y^\epsilon\cup \Delta\}$, and construct new multisets $\widehat{X}$ and $\widehat{Y}$ where we replace the copies of $\Delta$ with the corresponding locations on the diagonal which are closest to the point in to other matched off-diagonal point in the other persistence diagram. We then can construct $\widehat{\Mch}\subset\widehat{X}\times\widehat{Y}$ such that for each $p$, the $p$-costs for $\Mch$ and $\widehat{\Mch}$ are the same. Note that the number of pairs $(x,y)\in \widehat{\Mch}$ is finite, and $\frac{p'}{p}\leq 1$ and thus we can apply inequality \eqref{eq:useful} to show
	%$$\alpha_\Dpt = \left(\left(\death(\Dpt)- \death(\Mch(\Dpt))\right)^\power + \left(\birth(\Dpt)- \birth(\Mch(\Dpt))\right)^\power \right)^{\frac{1}{\power}}.$$
	\begin{align*}
		\cost_p(\Mch) & = \left(\sum\limits_{(x,y)\in\widehat{ \Mch}}
		|\death(x)- \death(y)|^\power + |\birth(x)- \birth(y)|^\power
		\right)^{\frac{1}{\power}}                                                                                                                                \\
		              & = \left(\sum\limits_{(x,y)\in\widehat{ \Mch}}
		|\death(x)- \death(y)|^\power + |\birth(x)- \birth(y)|^\power
		\right)^{\frac{\power'}{\power'\power}}                                                                                                                   \\
		              & \leq \left( \sum\limits_{(x,y)\in \widehat{\Mch}} |\death(x)- \death(y)|^{\power'} + |\birth(x)- \birth(y)|^{\power'} \right)^{1/\power'} \\
		              & =\cost_{p'}(\Mch)
	\end{align*}
	As $\Mch$ was an arbitrary matching this implies that $\cost_p(\Mch)\leq  \cost_{p'}(\Mch)$ for $p'<p$ and every matching between $X^\epsilon$ and $Y^\epsilon$. As we can couple the elements of the set of costs of all matchings, when we take the infimum over the set of all matchings we get $W_{p}(X^\epsilon,Y^\epsilon)\leq W_p (X^\epsilon,Y^\epsilon)$. The proof is completed by taking the limit as $\epsilon$ goes to zero.
\end{proof}

\section{Proof of Proposition \ref{prop:persistence_algorithm}}\label{appendix:persistence}
Here we give a proof sketch of the proposition. It follows directly from the algorithm for computing persistence diagrams so it will be obvious to experts, but we include the relevant details and references here. The key idea throughout is to construct the filtration incrementally, adding one simplex at a time. If the filtration is a total order, then this is unique. If it is a partial order, the insertion order will depend on the choice of extending the partial order to a total order. Throughout this section, unless explictly mentioned, we assume a total order.  %The first part of the proposition is with repect to the partition of simplices.

If a filtration is a total ordering, there is a unique partition of simplices into positive and negative simplices, where
	\begin{itemize}
		\item Positive simplices are those whose insertion generates a new homology class;
		\item Negative simplices are those whose insertion bounds an exiting homology class.
	\end{itemize}
Given a total ordering, there is a unique insertion order. By \cite{delfinado1993incremental}, the insertion of a simplex either creates a new class or bounds an existing one. The result follows. 
We remark that this is standard terminology within the literature for algorithms for computing persistent homology, e.g. \cite{edelsbrunner2000topological,zomorodian2005computing,cohen2006vines,EdeHar2010}. 

To show the remainder of the proposition, we outline the persistence algorithm. In \cite{zomorodian2005computing}, an algorithm for computing the summand decomposition is given in terms of reducing the boundary matrix into a specific reduced form -- called the column echelon form. Summarizing the algorithm, it arranges the columns (left-to-right) and rows (top-to-bottom) with respect to the order of filtration, so that there the restriction to a submatrix represents the boundary matrix of the complex at the corresponding step of the filtration. The matrix is then reduced left to right, where in each step, if the column's lowest non-zero entry cannot be zeroed out with the reduced previous columns or the column is zero, the algorithm moves to the next column.  After the matrix is reduced, the intervals can be read from this reduced matrix. Consider a column which corresponds to the simplex $\tau$. If it is non-zero, the row of the lowest non-zero entry in the column is called the \emph{pivot}. If $\sigma$ corresponds to the row of the pivot, we say that there is a pairing $(\sigma,\tau)$ and there is a corresponding finite interval $[f(\sigma),f(\tau))$. If it the column is zero and does not appear in an finite interval are unpaired positive simplices and so correspond to infinite intervals $[f(\tau),\infty)$.

Observe that the output of the algorithm depends only on the order in which the the simplices are processed, i.e. the ordering of the columns and rows of the boundary matrix. We conclude given a total ordering, the pairings are uniquely determined.% which is the second part of the proposition. 

\end{document}